\documentclass[11pt,a4paper,reqno]{amsart}
\usepackage{amsmath,amssymb,amsfonts,amsthm}
\usepackage{dsfont}
\usepackage[utf8]{inputenc}
\usepackage[T2A]{fontenc}
\usepackage[american]{babel}
\usepackage{amsopn}
\usepackage{bbm}
\usepackage{enumerate}
\usepackage{mathrsfs}
\usepackage{wasysym}

\usepackage{a4wide}
\usepackage{color}
\usepackage{pdfsync}
\usepackage{graphicx,subfigure}
\usepackage{hyperref}

\newtheorem{theorem}{Theorem}[section]
\newtheorem{lemma}[theorem]{Lemma}
\newtheorem{corollary}[theorem]{Corollary}
\newtheorem{proposition}[theorem]{Proposition}

\theoremstyle{definition}
\newtheorem{definition}[theorem]{Definition}

\theoremstyle{remark}
\newtheorem{remark}[theorem]{Remark}

\numberwithin{equation}{section}

\makeatletter

\@addtoreset{equation}{section}
\makeatother

\allowdisplaybreaks

\newcommand{\N}{{\mathbb N}}

\newcommand{\R}{{\mathbb R}}
\renewcommand{\C}{{\mathbb C}}

\newcommand{\K}{{\mathcal K}}
\newcommand{\M}{\widetilde{M}}
\newcommand{\eps}{{\varepsilon}}

\def\cal#1{\mathcal{#1}}

\renewcommand{\d}{\, {\rm d }}
\newcommand{\loc}{\rm loc}
\renewcommand{\div}{\operatorname{div}}
\newcommand{\curl}{\operatorname{curl}}
\newcommand{\dist}{\operatorname{dist}}
\newcommand{\supp}{\operatorname{supp}}
\renewcommand{\leq}{\leqslant}
\renewcommand{\geq}{\geqslant}

\title[Euler equations in a porous medium]{A homogenized limit for the 2D Euler equations in a perforated domain}
\author[M. Hillairet, C. Lacave \& D. Wu]{Matthieu Hillairet, Christophe Lacave \& Di Wu}

\address[M. Hillairet]{Institut Montpelli\'erain Alexander Grothendieck, CNRS, Univ. Montpellier, France}
\email{matthieu.hillairet@umontpellier.fr}

\address[C. Lacave]{Univ. Grenoble Alpes, CNRS, Institut Fourier, F-38000 Grenoble, France.}
\email{Christophe.Lacave@univ-grenoble-alpes.fr}

\address[D. Wu]{School of Mathematical Sciences, Peking University, Beijing 100871, P.R. China}
\email{wudi@math.pku.edu.cn}
\date{\today}

\begin{document}

\maketitle

\begin{abstract}
We study the motion of an ideal incompressible fluid in a perforated domain. The porous medium is composed of inclusions of size $a$ separated by distances $\tilde d$ and the fluid fills the exterior. We analyse the asymptotic behavior of the fluid when $(a,\tilde d) \to (0,0)$.

If the inclusions are distributed on the unit square, this issue is studied recently when $\frac{\tilde d}a$ tends to zero or infinity, leaving aside the critical case where the volume fraction of the porous medium is below its possible maximal value but non-zero. In this paper, we provide the first result in this regime. In contrast with former results, we obtain an Euler type equation where a homogenized term appears in the elliptic problem relating the velocity and the vorticity. 

Our analysis is based on the so-called method of reflections whose convergence provides novel estimates on the solutions to the div-curl problem which is involved in the 2D-Euler equations.
\end{abstract}



\section{Introduction}

For inviscid fluids in a perforated domain, the only mandatory boundary condition, known as the impermeability condition, is that the normal component of the velocity vanishes. However, the standard tools in the homogenisation framework were developed for the Dirichlet boundary condition. This explains that most papers have focused on viscous fluid models where we can assume the no-slip boundary condition: see \cite{Allaire90a,Allaire90b,LaMa, Mikelic91,Sanchez80,Tartar80} for incompressible Stokes and Navier-Stokes flows and \cite{Diaz99, FeireislLu, Lu, Masmoudi02esaim} for compressible Navier-Stokes systems. Among the exceptions, we mention \cite{Allaire91} where the Navier slip boundary condition is considered, but with a scalar friction function which tends to infinity when the size of the inclusions vanishes.

\medskip

Before the studies of the second author, the only articles which handle inviscid flows \cite{LionsMasmoudi,MikelicPaoli} consider a weakly nonlinear Euler flow through a regular grid (balls of radius $a$, at distance $a$ from one another). Using the notion of two-scale convergence, they recover a limit system which couples a cell problem with the macroscopic one, a sort of Euler-Darcy’s filtration law for the velocity.

For the full Euler equations, when the inclusions are regularly distributed on the unit square, the second author together with Bonnaillie-No\"el and Masmoudi treats the case where the inter-holes distance $\tilde d$ is very large or very small compared to the inclusion size $a$. In the dilute case, i.e. when $\frac{\tilde d}a$ tends to infinity, it is proved in \cite{BLM} that the limit motion is not perturbed by the porous medium, namely, we recover the Euler solution in the whole space. If, on the contrary, $\frac{\tilde d}a\to 0$, the fluid cannot penetrate the porous region, namely, the limit velocity verifies the Euler equations in the exterior of an impermeable square \cite{LM}. Therefore, the critical case where $\frac{\tilde d}a\to \bar{k}>0$ is not covered by the analysis developed in these two previous articles. Our goal here is to provide a first result in this very challenging regime.

\medskip

We give now in full details the problem tackled in this paper. Let $K_{PM}$ be a fixed compact subset of $\R^2$ and $\K$ be a connected and simply-connected compact subset of $[-1,1]^2$ such that $\partial \K$ is a $C^{1,\alpha}$ Jordan curve, for some $\alpha >0$. These two sets are arbitrary but fixed throughout the paper. We assume that the porous medium is contained in $K_{PM}$ and made of tiny holes with the following features:
\begin{itemize}
\item the number of holes is large and denoted by the symbol $N$ ;
\item each hole is of size $a >0$ and shape $\K$:
\begin{equation}\label{domain1}
\K_{\ell}^{a}:= x_{\ell}+ \tfrac a 2 \K,
\end{equation}
where the $N$ points $x_{\ell}$ are placed such that $\K_{\ell}^{a} \subset K_{PM}$ ;
\item the minimum distance between two centers $x_{\ell}$ is larger than $d>0.$
\end{itemize} 
We point out that $d$ denotes here the minimum distance {\it between centers}, but as we consider regimes where $d/a\gg 1$ (meaning that there exists an arbitrary large constant $C$ such that $d/a \geq C$ whatever the number of particles), the results would be the same considering $\tilde d$ the distance {\it between holes}, but it would complicate uselessly the analysis throughout this paper.

The fluid domain $\mathcal{F}_{N}$ is the exterior of these holes. Our purpose is to compute a homogenized system when the indicator function of the porous medium:
\begin{equation}\label{def-mu}
\mu := \sum_{\ell=1}^N \mathds{1}_{B(x_{\ell},a)}
\end{equation}
is close to a limit volume fraction $k$. We restrict to pointwise small volume fractions. Namely $k$ is assumed to belong to the following set:
\begin{equation}\label{compatibility}
\mathcal{FV}(\varepsilon_{0}) :=\{ k\in L^\infty(\R^2), \ \supp(k) \subset K_{PM}, \ \|k \|_{L^{\infty}(\mathbb R^2)} \leq \varepsilon_{0}^2 \}
\end{equation}
where $\varepsilon_{0} >0$ is a parameter which will be fixed later on sufficiently small. Consistently, we restrict to the case where $a/d \leq \varepsilon_0$.
To summarize, the domains considered in this paper satisfy:
\begin{equation} \label{main_assumption} \tag{$A_{\varepsilon_0}$}
\begin{aligned}
& \mathcal{F}_N := \R^2 \setminus \Big(\bigcup_{\ell=1}^N \K_{\ell}^a\Big),
\qquad \mathcal{F}_N^c \subset K_{PM} , \qquad
d=\min_{\ell\neq p} \dist \Big( x_{\ell} , x_{p} \Big) \geq \frac{a}{\varepsilon_{0}}.
\end{aligned}
\end{equation}
To illustrate the conventions above, consider for example the case where the holes $(\mathcal K^a_{\ell})_{\ell=1,\ldots,N}$ are spheres distributed periodically on an orthogonal lattice in the unit square $[0,1]^2.$ In this case, we have
\[
d \sim \dfrac{1}{2\sqrt{N}}.
\]
Our notations and assumptions correspond then to the critical regime:
\begin{itemize}
\item $a = \varepsilon/\sqrt{N}$ with $\varepsilon$ small but non zero,
for the discrete model
\item $k = \pi \varepsilon^2 \mathds{1}_{[0,1]^2}$ for the continuous one.
\end{itemize} 
Still in this periodic framework, previous analysis focused on 
\begin{itemize}
\item the dilute regime \cite{BLM} in which
\[
\lim_{N \to \infty} \sqrt{N} a = 0 \quad \text{ and } k \equiv 0 \ ;
\]
\item the dense regime \cite{LM} in which
\[
\lim_{N \to \infty} \sqrt{N} a = \dfrac{1}{2} \quad \text{ and } k = \dfrac{\pi}{4}\mathds{1}_{[0,1]^2} .
\]
\end{itemize}
So, in this periodic case and in the full generality, volume fractions $k$ verifying \eqref{compatibility} correspond to small data in the critical regime which is not covered by \cite{BLM,LM}. We remark also that the case $k \equiv 0$ previously studied is covered by our analysis.

\medskip

As is standard in the analysis of Euler equations in perforated domains, we divide the study in two steps. The first crucial step is to understand the elliptic problem
 which gives the velocity in terms of the vorticity (the div-curl problem). For instance, the two key properties in \cite{LM} are some estimates for the stationary div-curl problem. Herein also, the important novelty is a refined estimate for this elliptic problem that
we explain in a first part. This information is then plugged into the Euler equations in vorticity form in the second step of the analysis.

\subsection{Main result on the div-curl problem}

Any tangent and divergence free vector field in $\mathcal{F}_N$ can be written as the perpendicular gradient of a stream function $\psi_{N}$. When the vorticity $f$ of this vector-field is bounded, has compact support, and zero circulation is created on the boundaries, this stream function is computed as the unique (up to a constant) $C^1$ function solution to the following elliptic problem:
\begin{equation}\label{def-psiN}
 \Delta \psi_{N}=f \text{ in }\mathcal{F}_N,\ \lim_{|x|\to \infty} \nabla \psi_{N}(x)=0,\ \partial_{\tau} \psi_{N}=0 \text{ on } \partial\mathcal{F}_N,\ \int_{\partial \K_{\ell}^a} \partial_{n} \psi_{N} \d s = 0 \text{ for all } \ell.
\end{equation}
The main purpose of the following theorem is to show that -- in the asymptotic regime under consideration in this paper -- $\psi_{N}$ is close to $\psi_{c}$ the unique (up to a constant) $C^1$ function solving a homogenized problem.
It appears that this homogenized problem depends on a matrix $M_{\K}\in \mathcal{M}_{2}(\R)$ associated to the shape of $\K$, and reads:
\begin{equation}\label{def-psic}
\div \big[ ({\rm I}_{2}+kM_{\K}) \nabla \psi_{c}\big] =f \text{ in }\R^2,\qquad \lim_{|x|\to \infty} \nabla \psi_{c}(x)=0.
\end{equation} 
For instance, if $\K$ is the unit disk, then $M_{\K}=2\rm{I}_{2}$. In Section~\ref{sec_res1}, we show that we can compare the asymptotics of $\psi_{c}$ and $\psi_N$ to the solution of the Laplace problem in $\R^2$ with source term $f$ given by:
\begin{equation*}
\psi_0 (x) = \dfrac{1}{2\pi}\int_{\mathbb R^2} \ln(|x-y|) f(y)\d y,
\end{equation*}
in the sense that the differences $\psi_N(x) - \psi_0(x)$ and $\psi_c(x)- \psi_0(x)$ both converge to a constant when $|x| \to \infty.$ In order to define uniquely $\psi_{N}$ and $\psi_{c}$, we fix the unknown constants by imposing that
\[
\lim_{|x| \to \infty} \psi_{N}(x) - \psi_0(x) = \lim_{|x|\to \infty} \psi_c(x) - \psi_0(x) = 0.
\]
With these conventions, our main result is:
\begin{theorem}\label{main-elliptic}
Let $\varepsilon_{0}>0$ sufficiently small. For any $R_{f}>0$, $M_{f}>0$, $\eta \in (0,1)$, $p\in (1,\infty)$ and $\mathcal{O} \Subset \mathbb R^2$, there exists $C$ such that for any $k \in \mathcal{FV}(\varepsilon_0)$, any $\mathcal{F}_{N}$ verifying \eqref{main_assumption} and any $f$ satisfying ${\rm Supp}(f) \subset B(0,R_{f}) \cap \overline{\mathcal{F}_{N}}$ and $\|f\|_{L^{\infty}(\mathbb R^2)} \leq M_{f}$, the solution $\psi_{N}$ of \eqref{def-psiN} can be split into
 \[
 \psi_{N}=\psi_{c}+\Gamma_{1,N}+\Gamma_{2,N} 
 \] 
where 
 $\Delta \Gamma_{j,N}=0$ in $\R^2 \setminus K_{PM}$ for $j=1,2$ and 
\[ 
 \| \nabla \Gamma_{1,N} \|_{L^2(\mathcal{F}_N\cap \mathcal{O})}+ \| \Gamma_{2,N} \|_{L^2(\mathcal{F}_N\cap \mathcal{O})} 
 \leq 
 C
\left[ \left( \dfrac{a}{d}\right)^{3-\eta} +
\|\mu -k\|^{ \frac{p(1-\eta)}{p+2}}_{W^{-1,p}(\mathbb R^2)}
 + \|\mu -k\|^{\frac12}_{W^{-1,p}(\mathbb R^2)} + \|k\|_{L^{\infty}(\mathbb R^2)}^2
 \right].
 \]
\end{theorem}

 \begin{remark}\label{rem-elliptic}
We emphasize that this theorem implies that $\psi_{c}$ is a first-order approximation of $\psi_N$ in terms of the porous-medium volume-fraction. Given $k \in \mathcal{FV}(\varepsilon_0)$ the maximal porous-medium volume-fraction is related to $\|k\|_{L^{\infty}(\mathbb R^2)}$ while for the discrete counterpart, {\em i.e.} a fluid domain $\mathcal{F}_{N}$ verifying \eqref{main_assumption}, it is related to $(a/d)^{2}.$ Consequently, the remainder term:
 \[
\left( \dfrac{a}{d}\right)^{3-\eta} + \|k\|_{L^{\infty}(\mathbb R^2)}^2
 \]
 is superlinear in terms of $\varepsilon^2$, where
 \[
 \varepsilon:= \sqrt{\left( \dfrac{a}{d}\right)^{2} + \|k\|_{L^{\infty}(\mathbb R^2)}},
 \]
leaving the possibility to compare the first-order expansions of $\psi_c$ and $\psi_N$. We note also that the error term $\|\mu -k\|_{W^{-1,p}(\mathbb R^2)}$ corresponds to the replacement of a discrete problem by a continuous one and can be chosen arbitrary small for $N$ large enough and well-placed $(x_{\ell})_{\ell=1,\ldots,N}$. 
 
 We emphasize also that, via standard energy estimates, $\psi_{c}$ is indeed the leading term of the expansion because $
 \psi_{c} = \mathcal{O}(1)$ (see also Section~\ref{sec_expansion_Psic}).
The candidate $\psi_c$ is a better approximation than the solution $\psi_{0}$ of the elliptic problem without any influence of the porous medium:
\begin{equation} \label{def-psi0}
 \Delta \psi_{0} =f \text{ in }\R^2,\qquad \lim_{|x|\to \infty} \nabla \psi_{0}(x)=0.
\end{equation}
Indeed, writing $\Delta (\psi_{0}-\psi_{c})= \div (k M_{\K} \nabla \psi_{c})$ and performing standard energy estimates, we obtain that $\psi_{0}-\psi_{c}= \mathcal{O}(\varepsilon^2)$ hence
\[
\psi_{N} - \psi_{0} = \mathcal{O}(\varepsilon^2)\gg \psi_{N} - \psi_{c}.
\]
It is also a much better approximation than the solutions $\psi_{S}$ in the exterior of the impermeable square:
 \[
 \Delta \psi_{S}=f \text{ in } \R^2\setminus K_{PM},\ \lim_{|x|\to \infty} \nabla \psi_{S}(x)=0,\ \partial_{\tau} \psi_{S}=0 \text{ on } \partial K_{PM},\ \int_{\partial K_{PM} } \partial_{n} \psi_{S} \d s = 0 .
 \]
Indeed, we also have in this case $\psi_{0}-\psi_{S}= \mathcal{O}(1)$ hence
 \[
\psi_{N} - \psi_{S} = \mathcal{O}(1)\gg \psi_{N} - \psi_{c}.
\]
\end{remark}

\medskip

The starting point of the proof of this theorem consists in rewriting the elliptic problem \eqref{def-psiN} into
\begin{equation*}
 \Delta \psi_{N}=f \text{ in }\mathcal{F}_N,\ \lim_{|x|\to \infty} \nabla \psi_{N}(x)=0,\ \psi_{N}=\psi_{N,\ell}^* \text{ on } \partial \K_{\ell}^a,\ \int_{\partial \K_{\ell}^a} \partial_{n} \psi_{N} \d s = 0 \text{ for all } \ell,
\end{equation*}
where $(\psi_{N,\ell}^{*})_{\ell=1,\ldots,N}$ are $N$ unknown constants (note that this family of real numbers is also defined up to an additive constant). These constants can be seen as the Lagrange multipliers of the next flux-free condition. A first candidate to approximate $\psi_N$ is naturally $\psi_0.$ This candidate matches the pde in the fluid domain $\mathcal F_N,$ boundary condition at infinity, and flux conditions on the holes, but not the boundary condition on $\partial \mathcal F_N.$ So, we add a corrector to $\psi_0$ which cancels the non-constant part of $\psi_0$ on the $\K_{\ell}^{a}$. This corrector is computed by summing solutions to cell problems around each of the holes $\K_{\ell}^{a}$ as if it was alone.
Taking into account that the holes are small, we could choose as model cell problem the following one (where $\K^a:=a\K$): 
\begin{equation*}
 \Delta \psi =0 \text{ in }\mathbb R^2 \setminus \K^a ,\quad \psi(x)= A\cdot x + \psi^* \text{ on } \partial \K^a, \quad \lim_{|x|\to \infty} \psi(x)=0,
\end{equation*}
with $A \in \mathbb R^2$ a data representing the forcing by $\psi_0$ on the boundaries and $\psi^*$ an unknown constant. Up to a shift in space, we show in Section~\ref{sec_res1} that we can alternatively choose:
\begin{equation*}
 \Delta \psi =0 \text{ in } \mathbb R^2 \setminus \K^a, \quad \psi(x)=A \cdot x \text{ on } \partial \K^a,
 \quad \lim_{|x| \to \infty} \psi(x) = 0.
\end{equation*}
 Obviously, the solutions to these elementary problems do not take into account the other holes. So summing such solutions translated around the $(\K^a_{\ell})_{\ell=1,\ldots,N}$ we create again an error term in the boundary conditions on the holes. The strategy that we implement here is to introduce an iteration process in which we correct after each step the new error in the boundary conditions on the holes. 
This method is known as the "method of reflections" and has been widely studied in the context of elliptic problems (see \cite{LaurentLegendreSalomon,Hofer-Velazquez} for instance and \cite{Schubert} in the situation studied herein). It is recently adapted to the Stokes equations to study the effective viscosity problem by the first and last authors \cite{HillairetWu} (see also \cite{Niethammer-Schubert}).

\medskip

We point out that our elliptic problem \eqref{def-psiN} is also related to the perfect conductivity problem, namely when the conductivity tends to infinity (see the Appendix of \cite{BaoLiYin} for the link). In this context, there are many results in homogenization, and we refer to the recent paper by Bonnetier, Dapogny and Triki \cite{Bonnetier} for an overview of the literature. However, we did not find a result of the form of Theorem~\ref{main-elliptic}. In \cite{Gloria}, the author analyzes our elliptic problem \eqref{def-psiN} also by seeing it as a scalar version of a sedimentation problem. Asymptotics of the solution are obtained
in dimension $d >2$ when the positions of the particles (corresponding to the holes in our case) are given by a suitable hardcore point process. 
We guess that Theorem~\ref{main-elliptic} has its own interest and could be used in various problems (for instance in solid mechanics or electromagnetism). But, we restrict now to an application for the study of fluid motions.

\subsection{Application to the 2D Euler flows}

Even if the Euler equations is the oldest PDE, the study of this system is still a very active area of research, in mathematics as well as in engineering and physics, because it describes well the motion of incompressible fluids for a wide range of Reynolds numbers. In dimension two, the standard velocity formulation is equivalent to the vorticity formulation which reads
\begin{equation}\label{EulerN}
 \left\{
\begin{array}{ll}
 \partial_{t} \omega_{N} + u_{N}\cdot \nabla \omega_N=0, &\text{in }[0,T]\times \mathcal{F}_{N},\\
 \div u_{N} =0, \quad \curl u_{N} =\omega_{N}, &\text{in }[0,T]\times \mathcal{F}_{N},\\
\displaystyle u_{N}\cdot n\vert_{ \partial\mathcal{F}_{N}} =0, \quad \lim_{|x|\to \infty}u_{N}(\cdot,x)=0, &\text{on }[0,T],\\
\displaystyle\int_{\partial \K_{\ell}^a} u_{N} \cdot \tau \d s = 0, & \text{on }[0,T], \text{ for all } \ell=1,\dots, N,\\
 \omega_{N}(0, \cdot)=\omega_{0}, &\text{in } \mathcal{F}_{N},
\end{array}
 \right.
\end{equation}
with $\curl u_{N} = \partial_{1} u_{N,2}-\partial_{2} u_{N,1}$.
One of the main feature of \eqref{EulerN} is that it reduces to a transport equation
for the vorticity by the divergence free velocity $u_{N}$, where $u_{N}$ is computed from $\omega_{N}$ through a div-curl problem. 
The global well-posedeness of this equation -- in such exterior domains 
and $C^1_{c} (\R^2)$ initial data -- is established from a long time ago by Kikuchi \cite{Kikuchi} (see the textbook \cite{MajdaBertozzi} for more references).

The div-curl problem can be recast in terms of the stream function $\psi_{N}$ which is then the unique (up to a constant) solution of \eqref{def-psiN} with $f=\omega_{N}$. As in many papers on the 2D-Euler equations, once the properties of the operator which gives $u_{N}$ in terms of $\omega_{N}$ are analyzed -- which is exactly the purpose of Theorem~\ref{main-elliptic} -- one proves that $(\omega_{N},u_{N})$ is close to the solution $(\omega_{c},u_{c})$ of the following modified Euler system:
\begin{equation}\label{Eulerc}
 \left\{
\begin{array}{ll}
 \partial_{t} \omega_{c} + u_{c}\cdot \nabla \omega_{c}=0, &\text{in }[0,T]\times \R^2,\\
 \div u_{c} =0,\quad \curl(({\rm I}_{2}+k\widehat{M_{\K}}) u_{c}) =\omega_{c} &\text{in }[0,T]\times \R^2,\\
 \displaystyle \lim_{|x|\to \infty}u_{c}(\cdot,x)=0 &\text{on }[0,T],\\
 \omega_{c}(0, \cdot)=\omega_{0}, &\text{in } \R^2,
\end{array}
 \right.
\end{equation}
where $\widehat{M_{\K}}$ is defined in terms of $M_{\K}=(m_{i,j})_{i,j=1,2}$ as follows:
\[
\widehat{M_{\K}}:=
\begin{pmatrix}
 m_{22} & -m_{21}\\
 -m_{12} & m_{11}
\end{pmatrix}.
\]
The homogenized system \eqref{Eulerc} is also a transport equation for the vorticity $\omega_c$ by the divergence free vector field $u_c$, but $u_c$ is now related to $\omega_{c}$ through a modified div-curl problem. This new system is reminiscent of \eqref{def-psic} with $f=\omega_{c}$. Indeed, since $u_c$ is divergence-free, it reads again $u_{c}=\nabla^\perp \psi_{c}=(-\partial_{2}\psi_{c} , \partial_{1} \psi_{c})^T$ which implies that
\[
\curl(({\rm I}_{2}+k\widehat{M_{\K}}) \nabla^\perp \psi_{c})=\div(({\rm I}_{2}+kM_{\K}) \nabla \psi_{c}).
\]

\medskip

Our main result concerning the Euler equations splits in two parts: a well-posedness result for \eqref{Eulerc} and a stability estimate between the solution to \eqref{Eulerc} and the solution to the initial Euler problem in a perforated domain.

\begin{theorem}\label{main-Euler}
Let $\omega_{0}\in C^1_{c}(\R^2)$ such that $\supp \omega_{0} \Subset \R^2 \setminus K_{PM}$ and let $\delta \in (0,\dist (\supp\omega_{0},K_{PM}))$. There exists $\varepsilon_{0}>0$ such that the following holds true:

\begin{enumerate}
\item For any $k \in \mathcal{FV}(\varepsilon_0)$
there exists $T_{k}\in (0,+\infty]$ and a unique $\omega_{c}\in C^1([0,T_{k}]\times \R^2)$ solution to \eqref{Eulerc} such that $\dist(\supp \omega_{c}(t,\cdot),K_{PM})\geq \delta$ for any $t\in [0,T_{k}]$.

\item For any $T\leq T_{k}$, $\eta \in (0,1)$, $p\in (1,2)$, there exists $C(T,\eta,p)$ such that, for any $\mathcal{F}_{N}$ verifying \eqref{main_assumption}, the unique solution of \eqref{EulerN} with initial datum $\omega_{0}$ satisfies for all $t\in [0,T]$:
\begin{equation*}
\|(\omega_N-\omega_c)(t,\cdot)\|_{L^{\infty}(\R^2)}\leq 
C(T,\eta,p)
\left[ \left( \dfrac{a}{d}\right)^{3-\eta} +
\|\mu -k\|^{ \frac{p(1-\eta)}{p+2}}_{W^{-1,p}(\mathbb R^2)}
 + \|\mu -k\|^{\frac12}_{W^{-1,p}(\mathbb R^2)} + \|k\|_{L^{\infty}(\mathbb R^2)}^2
 \right].
\end{equation*}
Moreover, for any bounded open set $\mathcal{O} \Subset \R^2\setminus K_{PM}$, there exists $C(\mathcal{O},T,\eta,p)>0$ such that for all $t\in [0,T]$:
\begin{equation*}
\|(u_N-u_c)(t,\cdot)\|_{L^\infty(\mathcal{O})}\leq C(\mathcal{O},T,\eta,p)
\left[ \left( \dfrac{a}{d}\right)^{3-\eta} +
\|\mu -k\|^{ \frac{p(1-\eta)}{p+2}}_{W^{-1,p}(\mathbb R^2)}
 + \|\mu -k\|^{\frac12}_{W^{-1,p}(\mathbb R^2)} + \|k\|_{L^{\infty}(\mathbb R^2)}^2
 \right].
 \end{equation*}
\end{enumerate}
 \end{theorem}

Of course, if $T_{k}<+\infty$, we should choose $T=T_{k}$ in the second statement. During the proof, we compute also a stability estimate between the flow maps which correspond respectively to $u_{c}$ and $u_{N}$. To avoid additional definitions, we do not include this result in the statement of our main theorem. 
We mention here that this latter stability result follows mainly from the bootstrap argument developed by the second author together with Ars\'enio and Dormy \cite{ADL}. It is based on lipschitz estimates for $u_{c}-u_{N}$. Such a $W^{2,\infty}$ estimate for the stream functions is well beyond the content of Theorem~\ref{main-elliptic}. This reason motivates that we only get estimates on $[0,T_{k}]$, {\em i.e.} before that the homogenized vorticity reaches $\supp k$. Even if there is no vorticity in the vicinity of the porous medium, we recall that the velocity $u_{c}$ is highly affected by $k$ through a non-local operator. In particular, Remark~\ref{rem-elliptic} can be adapted here to state that the solution of the Euler equations in the whole plane or outside $K_{PM}$ is a worse approximation of $(\omega_{N},u_{N})$ than $(\omega_{c},u_{c})$.

We note that the above result is by nature slightly different from the usual results on the asymptotic behavior in perforated domains (as in the articles listed in the introduction and the references therein). Often, the justification that a homogenization problem is a good approximation reads as a weak or strong compactness theorem as $N\to \infty$, and in general, $\dot H^1_{\loc}$ estimates (like for $\Gamma_{1,N}$) is enough to have a global compactness result without assumption on the support of the vorticity. Unfortunately, we cannot ensure that every right-hand side terms in Theorems~\ref{main-elliptic} and \ref{main-Euler} tends to zero for $\mu\rightharpoonup k$. Here, our justification reads as the identification of the leading term with respect to powers of $\varepsilon$. Nevertheless, in classical literature, some weak topologies do not allow to give a precise estimate of the error between $\omega_{c}$ and $\omega_{N}$, hence we think that such a statement is interesting at the practical point of view.

\medskip

The remainder of this paper is composed of two parts. The following section deals with the elliptic estimates, namely proves Theorem~\ref{main-elliptic}. The application to the 2D-Euler equations is performed in Section~\ref{sec-Euler}.

\subsection*{Notations}

Below, we use standard notations for lebesgue/sobolev spaces. 
We also denote by $\dot{W}^{1,p}(\mathbb R^2)$ the classical homogeneous sobolev spaces. We shall also use $\dot{H}^1(\mathbb R^2)$ for $\dot{W}^{1,2}(\mathbb R^2).$

\section{Elliptic estimate} \label{sec_res1}

In the whole section, $R_{f}>0$, $M_{f}>0$ and $\varepsilon_0 >0$ are fixed. We fix also a homogenized volume fraction $k\in \mathcal{FV}(\varepsilon_0)$ and a porous medium $\mathcal{F}_{N}$ verifying \eqref{main_assumption}. 
We look for the restrictions on $\varepsilon_0$ such that Theorem~\ref{main-elliptic} holds true. To this end we fix again a source term $f$ so that 
${\rm Supp}(f) \subset B(0,R_{f})\cap \overline{\mathcal{F}_{N}}$ 
and $\|f\|_{L^{\infty}(\mathbb R^2)} \leq M_{f}$. We emphasize that, in this section, all constants can depend implicitly on $K_{PM}$ and $\K$, namely $C(q,\varepsilon_{0},R_{f})=C(q,\varepsilon_{0},R_{f}, K_{PM}, \K)$.

\medskip

We split the proof of Theorem~\ref{main-elliptic} into three parts. In the first part, we focus on the problem in the perforated domain \eqref{def-psiN}. We recall existence/uniqueness properties for this problem and provide an approximation of the solution via the method of reflections. In the second part, we focus on the homogenized problem \eqref{def-psic}. We again consider the well-posedness issue for this problem and provide an expansion of the solution with respect to the homogenized volume fraction $k.$
In the last part, we compare the solutions to \eqref{def-psiN} and to \eqref{def-psic} through the provided approximations.

\medskip

Before going into the core of the section, we recall basics on the resolution of the Laplace problem in the absence of holes \eqref{def-psi0}. Since $f$ has compact support, the unique (up to a constant) $C^1$ solution $\psi_0$ is given by the integral formula:
\begin{equation}\label{form-psi0}
\psi_0 (x) = \dfrac{1}{2\pi}\int_{\mathbb R^2} \ln(|x-y|) f(y)\d y.
\end{equation}
From this explicit formula, it is easy to derive the following standard estimates\footnote{We refer for instance to \cite[App. 2.3]{MarPul}.}:
\begin{itemize}
 \item $\psi_0$ is harmonic in the exterior of $B(0,R_{f})$ and behaves at infinity as follows
\begin{equation}\label{psi0-inf}
 \psi_0(x) = \frac{\int f}{2\pi} \ln |x| + \mathcal{O}\Big(\frac1{|x|}\Big) , \quad \nabla\psi_0(x) = \frac{\int f}{2\pi} \frac{x}{|x|^2} + \mathcal{O}\Big(\frac1{|x|^2}\Big) \, ; 
\end{equation}
 \item $\nabla \psi_0$ is uniformly bounded:
\begin{equation}\label{psi0-bd}
 \nabla \psi_0 (x) = \dfrac{1}{2\pi}\int_{\mathbb R^2} \frac{x-y}{|x-y|^2} f(y)\d y, \quad \| \nabla \psi_0 \|_{L^\infty(\mathbb R^2)} \leq C \| f\|_{L^1(\mathbb R^2)}^{1/2}\| f\|_{L^\infty(\mathbb R^2)}^{1/2},
\end{equation}
 with $C$ independent of $f$ ;
 \item $\nabla \psi_0$ is continuous and almost lipschitz:
\begin{equation}\label{psi0-lips}
| \nabla \psi_0 (x) - \nabla \psi_0 (y) | \leq C (\| f\|_{L^1(\mathbb R^2)} + \| f\|_{L^\infty(\mathbb R^2)}) h(|x-y|), \quad h(r)=r\max( -\ln r , 1),
\end{equation}
 with $C$ independent of $f$.
\end{itemize}

\subsection{Approximation of $\psi_N$ via the method of reflections}

We start by recalling the existence theory for \eqref{def-psiN}. 
At first, we note that the boundary conditions on $\partial \mathcal F_N$ impose that $\psi_N$ is constant on each connected component of $\partial \mathcal F_N. $
Consequently, we may rewrite \eqref{def-psiN} as:
there exist constants $(\psi^{*}_{N,\ell})_{\ell=1,\ldots,N}$ such that:
\begin{equation} \label{eq_defpsiN-2}
 \Delta \psi_{N}=f \text{ in }\mathcal{F}_N,\ \lim_{|x|\to \infty} \nabla \psi_{N}(x)=0,\ \psi_{N}=\psi_{N,\ell}^* \text{ on } \partial \K_{\ell}^a,\ \int_{\partial \K_{\ell}^a} \partial_{n} \psi_{N} \d s = 0 \text{ for all } \ell.
\end{equation}

Existence and uniqueness (up to a constant) of a $C^1$ solution $\psi_N$ follows from the arguments of \cite[Section 1]{Kikuchi} (see also \cite[(2.2)]{Kikuchi}). By standard ellipticity arguments -- and because $\partial\K\in C^{1,\alpha}$ with $\alpha>0$ -- we note that this solution satisfies 
$\psi_N \in W^{2,\infty}_{\loc}(\mathcal F_N)\cap C^1(\overline{\mathcal F_N})$. As $\psi_{N}$ is harmonic in the exterior of $B(0,R_{f})\cup K_{PM}$, it is simple\footnote{Indeed, the maps $z=x_{1}+ix_{2} \to \partial_{1} \psi_{N}-i\partial_{2}\psi_{N}$ is an holomorphic function which admits a Laurent expansion at infinity, where we compute that the leading term is $\frac1{2\pi z} ( \int_{\partial B(0,r)} \nabla \psi_{N}\cdot n + i\int_{\partial B(0,r)} \nabla \psi_{N}\cdot n^\perp)= \frac1{2\pi z} \int_{\mathcal{F}_{N}} f$ (for any $r> R_{f}$).} to obtain that $\psi_{N}$ behaves at infinity like $\psi_{0}$:
\begin{equation}\label{psiN-inf}
 \psi_N(x) = \frac{\int f}{2\pi} \ln |x| + \mathcal{O}\Big(\frac1{|x|}\Big) , \quad \nabla\psi_N(x) = \frac{\int f}{2\pi} \frac{x}{|x|^2} + \mathcal{O}\Big(\frac1{|x|^2}\Big) \, ,
\end{equation}
up to fix that $ \psi_N(x) - \frac{\int f}{2\pi} \ln |x| \to 0$ at infinity (we recall that $\psi_N$ is defined up to constant). It is then obvious that 
\[
\psi_N - \psi_0 \in L^p(\mathcal F_N), \; \forall \, p>2 ; 
\qquad
\nabla (\psi_N - \psi_0) \in L^q(\mathcal F_N), \; \forall\, q>1.
\]

\medskip

We define now the auxiliary fields (the so-called {\it reflections}) which are summed to provide the approximation of $\psi_{N}$. For this, let first note that the one-obstacle version of \eqref{eq_defpsiN-2} with a linear forcing $x \mapsto A\cdot x$ 
(obtained by linearizing the boundary condition coming from the lifting term $\psi_0$) on the boundary reads:
\begin{equation} \label{eq_reflection1}
 \Delta \psi =0 \text{ in }\mathbb R^2 \setminus \K,\quad \psi(x)= A\cdot x + \psi^* \text{ on } \partial \K, \quad\lim_{|x|\to \infty} \psi(x)=0.
\end{equation}
Similarly to \eqref{psiN-inf}, for any $\gamma\in \R$, we note that the unique solution to 
\begin{equation}\label{eq_reflection2}
\Delta \tilde\psi(x)=0 \text{ in }\R^2 \setminus {\K},\ \lim_{|x|\to \infty} \nabla \tilde\psi(x)=0,\ \tilde\psi(x) = A\cdot x \text{ on }\partial {\K} ,\ \int_{\partial \K} \partial_n \tilde\psi \d s = \gamma,
\end{equation}
enjoys the asymptotic expansion 
\begin{equation}\label{phi-inf}
\tilde\psi(x) = \dfrac{\gamma}{2\pi} \ln |x| + cstt + \mathcal{O}\left( \dfrac{1}{|x|}\right).
\end{equation}
Setting $\gamma =0$, $\psi^*=-cstt$ and $\psi=\tilde\psi+\psi^*$, \eqref{eq_reflection2} with the assumption $\int_{\partial \K} \partial_n \tilde\psi \d s =0 $ is equivalent to \eqref{eq_reflection1} where $\psi^*$ is uniquely determined.
In order to get rid of the constant $\psi^*$ we introduce the following definition:
\begin{definition}
We say that the domain $\widetilde{\K}$ is {\em well-centered} if, whatever the value of 
$A \in \mathbb R^2$ there exists a unique solution to 
\begin{equation} \label{reflection}
\Delta V^1[A](x)=0 \text{ in }\R^2 \setminus \widetilde{\K},\ \lim_{|x|\to \infty} V^1[A](x)=0,\ V^1[A](x) = A\cdot x \text{ on }\partial \widetilde{\K} .
\end{equation}
\end{definition}

\begin{remark}\label{rem-disk1}
A disk around the origin $\K=\overline{B(0,1)}$ is well centered and we also have an explicit formula for $V^1[A]$:
 \[
 V^1[A](x)= \dfrac{A\cdot x}{|x|^2} \text{ in } \R^2\setminus \overline{B(0,1)}.
\]
At the opposite, the disk $\K=\overline{B(x_{0},1)}$ where $x_{0}\neq 0$ is not well centered because the bounded solution at infinity of 
\[
\Delta V^1[A](x)=0 \text{ in }\R^2 \setminus \widetilde{\K},\ V^1[A](x) = A\cdot x \text{ on }\partial \widetilde{\K} 
\]
is 
\[
 V^1[A](x)= \dfrac{A\cdot (x-x_{0})}{|x-x_{0}|^2} + A\cdot x_{0} \text{ in } \R^2\setminus \overline{B(x_{0},1)},
\]
which cannot vanishes at infinity for every $A$.
\end{remark}

We prove in the following lemma that, up to shift a little the domain $\K$, we can assume that $\K$ is well-centered and that $V^1[A]$ behaves at infinity as in the case of the disk $\overline{B(0,1)}$:

\begin{lemma}\label{V1-decay}
Let $\K$ be a connected and simply-connected compact set of $\R^2$ whose boundary $\partial\K$ is a $C^{1,\alpha}$ Jordan curve (with $\alpha>1$). There exists a unique $c_{\K}$ in the convex hull of $\K$ such that, for any $A \in \mathbb R^2$, there exists a unique $C^1$ solution to \eqref{reflection} with $\widetilde{\K} = \K-c_{\K}$. Moreover, there exists a matrix $\widetilde{M_{\K}}\in \mathcal{M}_{2}(\R)$ and bounded vector fields $(h_{m})_{m=(m_1,m_2) \in \N^2}$ on $\R^2\setminus \tilde{\K}$ which depend only on the shape of $\K$ such that we have for all $x\in \R^2\setminus (\widetilde{\K}\cup B(0,1))$:
 \[
 \partial_ m V^1[A](x)= \partial_m \dfrac{(\widetilde{M_{\K}}A)\cdot x}{|x|^2} + A\cdot \frac{h_{m}(x)}{|x|^{2+m_1+ m_2 }} \quad\text{for any } m \in \N^2.
 \]
\end{lemma}

\begin{proof}
Let $A\in \R^2$ fixed. Setting $\phi(x) = V^1[A](x-c_{\K})$, we look for a condition on $c_{\K}$ such that the following problem is well-posed
\[
\Delta \phi(x)=0 \text{ in }\R^2 \setminus \K,\ \lim_{|x|\to \infty} \phi(x)=0,\ \phi(x) = A\cdot (x-c_{\K}) \text{ on }\partial \K .
\]

Defining $\psi = \phi - (A\cdot (x-c_{\K}))\chi(x)$ with $\chi$ a convenient cutoff function, it is clear from the Dirichlet Laplace problem in exterior domains that there exists (for any $c_{\K}$) a unique $C^1$ solution such that 
\[
\Delta \phi(x)=0 \text{ in }\R^2 \setminus \K,\ \lim_{|x|\to \infty} \nabla\phi(x)=0,\ \phi(x) = A\cdot (x-c_{\K}) \text{ on }\partial \K , \ \int_{\partial \K} \partial_n \phi \d s = 0,
\]
and we provide now an explicit formula in terms of Green's function, from where we will find $c_{\K}$ and the asymptotic behavior.

As explained above (see\eqref{phi-inf}), the condition $\int_{\partial \K} \partial_n \phi \d s = 0$ implies that (for any $c_{\K}$) we have the following expansion at infinity
\[
\nabla \phi(x) = \mathcal{O}\Big(\frac1{|x|^2}\Big) \quad \text{and} \quad \phi(x) = \mathcal{O}(1).
\]

Identifying $\R^2=\C$, by the Riemann mapping theorem, we consider the unique $\cal T$ biholomorphism from $\R^2\setminus \K$ to $\R^2\setminus \overline{B(0,1)}$ which verifies $\cal T(\infty)= \infty$ and $\cal T'(\infty)\in \R^+_{*}$, which reads as
 \[
 \cal T(z) = \beta z + g(z)
 \]
 for $\beta \in \R^+_{*}$ and $g$ a bounded holomorphic function. It is then well known that we can express the Dirichlet Green's function in terms of $\cal T$:
 \[
 G(x,y)=\frac1{2\pi}\ln \frac{|\cal T(y)-\cal T(x)|}{|\cal T(y)-\cal T(x)^*| |\cal T(x)|},\quad \text{ where } \xi^*=\frac{\xi}{|\xi|^2}.
 \]
 We refer for instance to \cite{ILL} where such a formula was used in the context of the Euler system. In particular, we note that for $x\in \R^2\setminus \K$ fixed, we have the following behavior when $y\to \infty$
 \[
 G(x,y) = \mathcal{O}(1) \quad \text{and} \quad \nabla_{y} G(x,y) =\frac{D \cal T^T(y)}{2\pi} \Big( \frac{\cal T(y)-\cal T(x)}{|\cal T(y)-\cal T(x)|^2} - \frac{\cal T(y)-\cal T(x)^*}{|\cal T(y)-\cal T(x)^*|^2} \Big) = \mathcal{O}\Big(\frac1{|y|^2}\Big).
\]

As $\cal T$ maps $\partial \K$ to $\partial B(0,1)$, a parametrization of the boundary $\partial K$ is given by $t\mapsto \cal T^{-1}\begin{pmatrix}\cos t \\ \sin t\end{pmatrix}$, hence a tangent vector at the point $y=\cal T^{-1}\begin{pmatrix}\cos t \\ \sin t\end{pmatrix}$ is given by 
\[
D\cal T^{-1}\begin{pmatrix}\cos t \\ \sin t\end{pmatrix} \begin{pmatrix}-\sin t \\ \cos t\end{pmatrix}= \Big(D \cal T(y)\Big)^{-1} \cal T^\perp(y) =\frac1{\det D\cal T(y)}D \cal T^T(y) \cal T^\perp(y),
\]
 where we have used that $D\cal T$ is under the form $\begin{pmatrix} a & b \\ -b & a \end{pmatrix}$ (Cauchy-Riemann equations). Using again that $\cal T(y)\in \partial B(0,1)$ and the form of $D\cal T$, we deduce that the outer normal of $\R^2\setminus \K$, denoted by $n$, is
 \[
 n(y)=- \frac{D \cal T^T(y) \cal T(y)}{|D \cal T^T(y) \cal T(y)|}= -\frac{D \cal T^T(y) \cal T(y)}{\sqrt{\det D\cal T(y)}}.
 \]
 Thanks to the decay properties of $G$ and $\phi$, this function $G$ allows us to derive the following representation formula for $\phi$:
\begin{align*}
\phi(x) =& \int_{\partial \K}(A\cdot (y-c_{\K})) (\nabla_{y} G(x,y)\cdot n(y)) \d \sigma(y) \\
 =& \frac A{2\pi} \cdot \int_{\partial \K} (y-c_{\K}) \Bigg[ \Bigg( D \cal T^T(y) \Big( \frac{\cal T(y)-\cal T(x)}{|\cal T(y)-\cal T(x)|^2} - \frac{\cal T(y)-\cal T(x)^*}{|\cal T(y)-\cal T(x)^*|^2} \Big) \Bigg)\cdot \Bigg(- \frac{D \cal T^T(y) \cal T(y)}{\sqrt{\det D\cal T(y)}} \Bigg) \Bigg] \d \sigma(y)\\
 =& -\frac A{2\pi} \cdot \int_{\partial \K} (y-c_{\K}) \Big( \frac{1-\cal T(x) \cdot \cal T(y)}{|\cal T(y)-\cal T(x)|^2} - \frac{1-\cal T(x)^*\cdot \cal T(y)}{|\cal T(y)-\cal T(x)^*|^2} \Big) \sqrt{\det D\cal T(y)} \d \sigma(y).
\end{align*}
 Now we use that $\cal T(x) = \beta x + \mathcal{O}(1)$ and again that $\cal T(y)\in \partial B(0,1)$ to get 
\begin{align*}
\phi(x) =& -\frac A{2\pi} \cdot \int_{\partial \K} (y-c_{\K}) \Big(-1 -2 \frac{x\cdot \cal T(y)}{\beta|x|^2} +\mathcal{O}\Big(\frac1{|x|^2} \Big) \Big)\sqrt{\det D\cal T(y)} \d \sigma(y)\\
 =&\frac A{2\pi} \cdot \Bigg( \int_{\partial \K} y \sqrt{\det D\cal T(y)} \d \sigma(y) - c_{\K}\int_{\partial \K} \sqrt{\det D\cal T(y)} \d \sigma(y)\Bigg) \\
 &+
 \frac A{\pi} \cdot \int_{\partial \K} (y-c_{\K}) \frac{x\cdot \cal T(y)}{\beta|x|^2} \sqrt{\det D\cal T(y)} \d \sigma(y) + A\cdot \mathcal{O}\Big(\frac1{|x|^2}\Big).
\end{align*}
It is then obvious that $\phi(x)\to 0$ at infinity if and only if
\[
c_{\K}= \frac{ \int_{\partial \K} y \sqrt{\det D\cal T(y)} \d \sigma(y)}{ \int_{\partial \K} \sqrt{\det D\cal T(y)} \d \sigma(y)} 
\]
which belongs to the convex hull of $\K$. 

Setting
\[
\widetilde{M_{\K}} = \frac1{\beta\pi} \Big(\int_{\partial \K} (y_{j}-c_{\K j}) \cal T_{i}(y) \sqrt{\det D\cal T(y)} \d \sigma(y) \Big)_{i,j}
\] 
gives the expansion of $V^1[A]$, because it is clear that
\[
V^1[A](x)= \phi (x+c_{\K}) = \dfrac{(\widetilde{M_{\K}}A)\cdot (x+c_{\K})}{|x+c_{\K}|^2} + A\cdot \mathcal{O}\Big(\frac1{|x|^2}\Big) = \dfrac{(\widetilde{M_{\K}}A)\cdot x}{|x|^2} + A\cdot \mathcal{O}\Big(\frac1{|x|^2}\Big)
\]
at infinity whereas $h_{0}(x)$ is bounded in any bounded subset of $\R^2\setminus (\widetilde{\K}\cup B(0,1))$.

Like for \eqref{psiN-inf}, we notice that the maps $z=x_{1}+ix_{2} \to \partial_{1} V^1[A]-i\partial_{2}V^1[A]$ admits a Laurent expansion at infinity, which is compatible with the previous expansion only if 
\[\partial_{i} \Big( V^1[A](x) - \dfrac{(\widetilde{M_{\K}}A)\cdot x}{|x|^2}\Big)=A\cdot\mathcal{O}(1/|x|^3) \quad \text{for } i=1,2.\]
By the Laurent series, we directly conclude that the decomposition $\partial_{m} \Big( V^1[A](x) - \dfrac{(\widetilde{M_{\K}}A)\cdot x}{|x|^2}\Big)=A\cdot\mathcal{O}(1/|x|^{2+m_{1}+m_{2}})$ holds true for any $m\in \N^2$.
\end{proof}

\begin{remark}\label{rem-disk}
In the case of a circular hole centered at the origin $\K=\overline{B(0,1)}$, we notice that $\cal T={\rm Id}$ gives $c_{\K}=0$ and $\widetilde{M_{\K}}={\rm I}_{2}$, which corresponds to Remark~\ref{rem-disk1}.
\end{remark}

From now on, we assume further that $\K$ is well-centered, namely
\[
\int_{\partial \K} y \sqrt{\det D\cal T(y)} \d \sigma(y) = 0.
\]
 We emphasize that this assumption is harmless for the computations below.
Indeed, the content of Lemma~\ref{V1-decay} yields that this assumption amounts to shift the origin of the frame in which the set $\K$ is defined. However, the necessary shift maps the origin into a point
inside the convex hull of $\K$ (and thus in the square $[-1,1]^2$ like $\K$) while the distance between the (scaled) holes of the porous medium $(\K^{a}_{\ell})_{\ell=1,\ldots,N}$ is much larger than the hole width. Hence, changing the origin
for the point that makes the set $\K$ well-centered does not change our results
on the method of reflections below. Consequently, we avoid tildas over sets ${\K}$ from now on.

\medskip

Given $a >0$, $a\K$ is also well centered, so for any $A \in \mathbb R^2$, there exists a unique solution $V^{a}[A]$ to 
\begin{equation} \label{reflection-a}
\Delta V^a[A](x)=0 \text{ in }\R^2 \setminus (a\K),\ \lim_{|x|\to \infty} V^a[A](x)=0,\ V^a[A](x) = A\cdot x \text{ on }\partial (a\K) ,
\end{equation}
which clearly verifies the scaling law
\begin{equation}\label{scaling-V}
 V^a[A](x)=a V^1[A](x/a).
\end{equation}
Let us note that the behavior of $V^a[A](x)$ for $|x|\geq d$, {\em i.e.} when $|x|/a\geq d/a\gg 1$, is given by Lemma~\ref{V1-decay}: 
\[
 V^a[A](x)= a^2 \dfrac{(\M_{\K}A)\cdot x}{|x|^2} + a^2 A\cdot \frac{h_{0}(x/a)}{|x|^2} = a^2 \dfrac{(\M_{\K}A)\cdot x}{|x|^2} + \Big( \frac ad\Big)^2 A\cdot \mathcal{O}(1).
\]

\medskip

From the behavior of $\nabla V^1$ at infinity, it is also clear that
\[
\lim_{R\to \infty} \int_{B(0,R)} \partial_n V^a[A](x) \d s =0
\]
which means by harmonicity that
\begin{equation}\label{Va-circ}
\int_{\partial (a\K)} \partial_n V^a[A](x) \d \sigma = 0. 
\end{equation}

\medskip

We provide now an approximation of the solution $\psi_{N}$ {\em via} the following iterative process {\em i.e.}, the so-called "method of reflections". At first, we consider the Laplace solution in the absence of holes \eqref{form-psi0}:
\begin{equation} \label{eq_A_init}
\psi^{(0)} = \psi_0.
\end{equation}
As $f$ vanishes inside the holes (we recall that $f$ has support in $\overline{\mathcal F_N}$ by assumption), we get by Stokes theorem that $\int_{\partial \K_{\ell}^a} \partial_{n} \psi_{0} \d s = 0$ for all $\ell$. Therefore $\psi^{(0)}$ verifies every condition of \eqref{def-psiN} except that it is not constant on $\partial \K_{\ell}^a$:
\[
\psi^{(0)}(x) = \psi^{(0)}(x_{\ell}) + \nabla \psi_0(x_\ell)\cdot (x-x_{\ell}) + o(x-x_{\ell}) \quad \text{on } \partial \K_{\ell}^a.
\]
Hence, we use the reflections introduced in \eqref{reflection} to correct the main error:
\begin{align}
A_{\ell}^{(1)} &:= -\nabla \psi_0(x_\ell) , \ \ell=1,\ldots,N, \label{A1}\\
\psi^{(1)} &:=\psi^{(0)}+\phi^{(1)}, \text{ with }\phi^{(1)}:=\sum_{\ell=1}^N V^{a}[A_{\ell}^{(1)}](x-x_\ell) . \nonumber
\end{align}
By \eqref{Va-circ} and by harmonicity of $V^a$, we note that $\psi^{(1)}$ verifies again every condition of \eqref{def-psiN} except that it is still not constant on $\partial \K_{\ell}^a$, but the non constant part will be smaller due to the decay property of $V^{a}$ (see Lemma~\ref{V1-decay}):
\begin{align*}
\psi^{(1)}(x) &=\psi_0(x_\ell)+ \sum_{\lambda \neq \ell}V^{a}[A^{(1)}_{\lambda}](x-x_\lambda)+ o(x-x_{\ell}) \quad \text{on } \partial \K_{\ell}^a\\
&=\psi^{(1)}(x_{\ell})+ \Big( \sum_{\lambda \neq \ell}\nabla V^{a}[A^{(1)}_{\lambda}](x_\ell-x_\lambda)\Big)\cdot (x-x_{\ell}) + o(x-x_{\ell}) \quad \text{on } \partial \K_{\ell}^a.
\end{align*}
We iterate this procedure: for any $n \in \mathbb N$ assuming the approximate solution $\psi^{(n)}$ to be constructed, we define:
\begin{align}
A^{(n+1)}_{\ell} & := - \sum_{\lambda\neq \ell}\nabla V^{a}[A^{(n)}_{\lambda}](x_\ell-x_\lambda) , \ \ell=1,\ldots,N, \label{eq_A_iter}\\
\psi^{(n+1)} &:=\psi^{(n)}+\phi^{(n+1)}, \text{ with }\phi^{(n+1)}:=\sum_{\ell=1}^N V^{a}[A_{\ell}^{(n+1)}](x-x_\ell) ,\label{eq_phi_iter}
\end{align}
which satisfies
\begin{align*}
 \psi^{(n+1)}(x)&=\psi^{(n)}(x_\ell) + \sum_{\lambda\neq \ell}V^{a}[A^{(n+1)}_{\lambda}](x-x_\lambda) +o(x-x_\ell ) \quad \text{on } \partial \K_{\ell}^a\\
&=\psi^{(n+1)}(x_{\ell})+ \Big( \sum_{\lambda \neq \ell}\nabla V^{a}[A^{(n+1)}_{\lambda}](x_\ell-x_\lambda)\Big)\cdot (x-x_{\ell}) + o(x-x_{\ell}) \quad \text{on } \partial \K_{\ell}^a.
\end{align*}

\medskip

For technical purpose, we associate to the $(A^{(n)}_{\ell})_{\ell=1,\ldots,N}$ the following vector-field:
\begin{equation} \label{eq_Phi}
\Phi^{(n)}(x):=\frac{4}{\pi^2}\sum_{\ell=1}^{N}A^{(n)}_{\ell}\mathbbm{1}_{B(x_\ell,d/2)}(x),
\end{equation}
where $\mathbbm{1}_{B(x_\ell,d/2)}(x)$ is the indicator function of ${B(x_{\ell},d/2)}(x)$. As the disks are disjoint, we have:
\begin{equation} \label{eq_LpPhi}
\|\Phi^{(n)}\|_{L^p(\mathbb R^2)} = \left(\dfrac{4^{p-1}d^{2}}{\pi^{2p-1}} \sum_{\ell=1}^N |A_{\ell}^{(n)}|^p \right)^{\frac 1p}
\end{equation}
for arbitrary finite $p.$

\medskip

The main purpose of this section is to prove that this method of reflections converges and that it yields a good approximation of $\psi_N.$
To this end, we control at first the sequence of vectors $(A_{\ell}^{(n)})_{\ell=1,\ldots,N}$:

\begin{lemma}\label{Aapp}
Assume that $0<\eps_0<1/2$ and $q \in (1,\infty).$
There exists a constant $C_{ref}$ depending only on $q$ for which the sequence $((A_{\ell}^{(n)})_{\ell=1,\ldots,N})_{n\in \mathbb N^*}$ as defined by \eqref{A1}-\eqref{eq_A_iter} satisfies:
\[
\Big(\sum_{\ell=1}^N |A^{(n+1)}_{\ell}|^q\Big)^{1/q}\leq C_{ref}(q)\Big(\frac{a}{d}\Big)^{2(1-\frac{1}{q})}\left(1+\ln\left( \frac{1+2\eps_0}{1-2\eps_0}\right)\right)\Big(\sum_{\ell=1}^N |A^{(n)}_{\ell}|^q\Big)^{1/q}.
\]
\end{lemma}

\begin{proof}
Let $n\in \mathbb{N}^{*}$. By definition \eqref{eq_A_iter} of $A^{(n+1)}_{\ell}$, the explicit expansion of $V^1$ (see Lemma~\ref{V1-decay}) and the scaling law \eqref{scaling-V}, we have that:
\begin{equation}\label{A-decompo}
A^{(n+1)}_{\ell}=
- a^2\sum_{\lambda\neq \ell}\nabla \Big(\dfrac{(\M_{\K} A^{(n)}_{\lambda} )\cdot x}{|x|^2}\Big) \Big\vert_{x=x_{\ell}-x_\lambda} - a^3 \sum_{\lambda\neq \ell} \frac{h_{1}(\frac{x_{\ell}-x_\lambda}a)}{|x_{\ell}-x_\lambda|^3}A^{(n)}_{\lambda},
\end{equation}
where $h_{1}$ is a $2\times 2$ matrix whose the lines are $h_{(1,0)}^T$ and $h_{(0,1)}^T$.

We begin by the last sum which can be easily bounded thanks to a pseudo discrete Young's convolution inequality inspired of \cite{GVH}:
\[
\forall q\geq 1 \qquad \Bigg(\sum_{\ell}\Big(\sum_{\lambda}|a_{\ell \lambda}||b_{\lambda}|\Big)^q\Bigg)^{1/q} \leq \max\Big( \sup_{\ell}\sum_{\lambda}|a_{\ell \lambda }|,\ \sup_{\lambda}\sum_{\ell}|a_{\ell \lambda }|\Big) \Big(\sum_{\ell}|b_{\ell}|^q\Big)^{1/q}.
\]
Indeed, we recall that $h_{1}$ is bounded, that $B(x_{\ell},d/2)\cap B(x_\lambda,d/2)=\emptyset$ for all $\ell\neq \lambda$. So, the previous inequality for $b_{\lambda} = A^{(n)}_{\lambda}$ and $a_{\ell \lambda }=|x_{\ell}-x_\lambda|^{-3}$ for $\ell \neq \lambda$ (otherwise $a_{\ell \ell}=0$) gives
\begin{align}
a^3 \Bigg(\sum_{\ell=1}^N \Big( \sum_{\lambda\neq \ell} \frac{|A^{(n)}_{\lambda}| }{|x_{\ell}-x_\lambda|^3} \Big)^q\Bigg)^{1/q}
\leq& a^3 \sup_{\ell}\sum_{\lambda\neq \ell} \frac{1 }{|x_{\ell}-x_\lambda|^3} \Big(\sum_{\ell=1}^N |A^{(n)}_{\ell}|^q\Big)^{1/q} \nonumber\\
\leq& a^3 \Big(C d^{-2} \int_{B(0,d/2)^c} |x|^{-3}\d x \Big)\Big(\sum_{\ell=1}^N |A^{(n)}_{\ell}|^q\Big)^{1/q} \nonumber\\
\leq& C \Big(\frac ad \Big)^{3} \Big(\sum_{\ell=1}^N |A^{(n)}_{\ell}|^q\Big)^{1/q}. \label{A-2nd}
\end{align}

This estimate is enough for the second sum of \eqref{A-decompo}, but we note that the first term is more singular. Indeed, a similar argument would yield:
\[
a^2\Bigg(\sum_{\ell=1}^N \Big( \sum_{\lambda\neq \ell} \frac{|A^{(n)}_{\lambda}| }{|x_{\ell}-x_\lambda|^2} \Big)^q\Bigg)^{1/q}
\leq C \Big(\frac ad \Big)^{2} |\ln d| \Big(\sum_{\ell=1}^N |A^{(n)}_{\ell}|^q\Big)^{1/q}
\]
which tends to infinity when $a,d\to 0$ (even with $a/d=\varepsilon_{0}$ fixed). So we provide a finer estimate by rewriting the first term as a convolution in terms of $\Phi^{(n)}$ (see \eqref{eq_Phi}) in order to apply a Cald\'eron-Zygmund inequality.

\medskip

We note that $x \mapsto (\M_{\K} A^{(n)}_{\lambda})\cdot (x_{\ell}-x)/{|x_{\ell}-x|^2}$ is harmonic in $B(x_\lambda,{d}/{2})$ for any $\lambda\neq \ell$. Hence according to the mean-value formula, we have
\[
a^2\sum_{\lambda\neq \ell}\nabla \Big(\dfrac{(\M_{\K} A^{(n)}_{\lambda} )\cdot x}{|x|^2}\Big) \Big\vert_{x=x_{\ell}-x_\lambda} 
= \frac{4a^2}{\pi d^2}\sum_{\lambda\neq \ell}\int_{B(x_\lambda,d/2)}\nabla \Big(\dfrac{(\M_{\K} A^{(n)}_{\lambda} )\cdot x}{|x|^2}\Big) \Big\vert_{x=x_{\ell}-y} \d y.
\]
Similarly, we remark that, for arbitrary $y \in B(x_{\lambda},d/2)$ with $\lambda \neq \ell,$ the mapping $x \mapsto (\M_{\K} A^{(n)}_{\ell})\cdot(x-y)/{|x-y|^2} $ 
is harmonic on $B(x_{\ell},a).$ This yields:
\begin{align}
a^2\sum_{\lambda\neq \ell}\nabla \Big(\dfrac{(\M_{\K} A^{(n)}_{\lambda} )\cdot x}{|x|^2}\Big) \Big\vert_{x=x_{\ell}-x_\lambda} 
&=\dfrac{4}{\pi^2 d^{2}} \sum_{\lambda\neq \ell} \int_{B(x_{\lambda},d/2)} \int_{B(x_{\ell},a)} \nabla \Big(\dfrac{(\M_{\K} A^{(n)}_{\lambda} )\cdot x}{|x|^2}\Big) \Big\vert_{x=z-y} \d z\d y \nonumber\\
&=\dfrac{1}{d^{2}} \int_{B(x_{\ell},a)} \int_{\R^2\setminus B(x_{\ell},d/2)} \Big(D_{z}\dfrac{z-y}{|z-y|^2}\Big)^T \M_{\K} \Phi^{(n)}(y) \d y\d z ,\nonumber\\
&:= I_{\ell}+J_{\ell}. \label{decompo2}
\end{align}
We denoted here $D_z$ for the gradient w.r.t. variable $z$ and we have splitted eventually the integral as follows:
\begin{align*}
I_{\ell}& :=d^{-2}\int_{B(x_{\ell},a)}\int_{\R^2\setminus B(z,d/2)} \Big(D_{z}\dfrac{z-y}{|z-y|^2}\Big)^T \M_{\K} \Phi^{(n)}(y) \d y\d z , \\
J_{\ell}&:=d^{-2}\int_{B(x_{\ell},a)}\int_{\R^2\setminus B(x_{\ell},d/2)} \Big(D_{z}\dfrac{z-y}{|z-y|^2}\Big)^T \M_{\K} \Phi^{(n)}(y) \d y\d z\\
& \qquad -d^{-2}\int_{B(x_{\ell},a)}\int_{\R^2\setminus B(z,d/2)} \Big(D_{z}\dfrac{z-y}{|z-y|^2}\Big)^T \M_{\K} \Phi^{(n)}(y) \d y\d z.
\end{align*}
We deal with $I_{\ell}$ first. We notice that $I_{\ell}$ can be regarded as an integral of a convolution: 
\begin{equation} \label{eq_Fn}
I_{\ell}(x)=d^{-2} \int_{B(x_{\ell},a)}F^{(n)}(z)\d z \quad\text{where}\quad
F^{(n)}(z):=\int_{\R^2\setminus B(0,d/2)} \Big(D_y \dfrac{y}{|y|^2}\Big)^T \M_{\K} \Phi^{(n)}(z-y) \d y.
\end{equation}
By H\"older inequality, we get
\[
|I_{\ell}|\leq d^{-2}(\sqrt{\pi}a)^{\frac{2}{q'}}\|F^{(n)}\|_{L^q(B(x_{\ell},a))},
\]
where $q'$ is the conjugate exponent of $q$. On the first hand, this entails:
\begin{equation}\label{Ij}
\sum_{\ell=1}^{N}|I_{\ell}|^q\leq d^{-2q}(\sqrt{\pi}a)^{\frac{2q}{q'}}\|F^{(n)}\|^q_{L^q(\R^2)}.
\end{equation}
On the other hand, we apply that $F^{(n)}$ is defined by an integral operator with kernel $\mathcal{K}(x,y):=\mathds{1}_{|x-y|\geq d/2}D({x-y}/{|x-y|^2}).$ This kernel enjoys the Cald\'eron-Zygmund condition so that (recall \eqref{eq_LpPhi}):
\begin{align}
\|F^{(n)}\|_{L^q(\R^2)} &\leq C(q)\|\Phi^{(n)}\|_{L^q(\R^2)} \notag \\ \label{Fk}
& \leq C(q)d^{\frac{2}{q}}\big(\sum_{\ell=1}^N|A^{(n)}_{\ell}|^q\big)^{1/q}.
\end{align} 
Combining \eqref{Fk} with (\ref{Ij}) yields that
\begin{equation}\label{I-est}
 \Big(\sum_{\ell=1}^{N}|I_{\ell}|^q\Big)^{1/q}\leq C(q)\big(\frac{a}{d}\big)^{\frac{2}{q'}} \Big(\sum_{\ell=1}^N|A^{(n)}_{\ell}|^q\Big)^{1/q}.
\end{equation}

\medskip

Now, we turn to deal with $J_{\ell}$. At first, we notice that for any $\ell=1,\ldots,N$ and any $z\in B(x_{\ell},a)$, 
\[
B(z,d/2) \Delta B(x_{\ell},d/2) \subset B(z,d/2+a) \setminus B(z,d/2-a) 
\]
(where $\Delta$ represents the symmetric difference between sets)
which implies that
\[
|J_{\ell}|
\leq \dfrac{C}{d^2}\int_{B(x_{\ell},a)}\int_{B(z,\frac{d}{2}+a)\setminus B(z,\frac{d}{2}-a)}\big|\Phi^{(n)}(y)\big|\frac{1}{|z-y|^2}\d y\d z = \dfrac{C}{d^2}\int_{B(x_{\ell},a)}G^{(n)}(z) \d z
\]
where we denote
\begin{equation} \label{eq_Gn}
G^{(n)}(z) :=\int_{B(z,\frac{d}{2}+a)\setminus B(z,\frac{d}{2}-a)}\big|\Phi^{(n)}(y)\big|\frac{1}{|z-y|^2}\d y =\int_{B(0,\frac{d}{2}+a)\setminus B(0,\frac{d}{2}-a)}\big|\Phi^{(n)}(z-y)\big|\frac{1}{|y|^2}\d y.
\end{equation}
By H\"older inequality, we obtain as previously that:
\begin{equation}\label{Jl}
\sum_{\ell=1}^N|J_{\ell}|^q\leq C d^{-2q}a^{\frac{2q}{q'}}\|G^{(n)}\|^q_{L^q(\R^2)}.
\end{equation}
By the standard Young's convolution inequality, we get by \eqref{eq_LpPhi}:
\begin{align*}
\|G^{(n)}\|_{L^q(\R^2)}\leq& \|\Phi^{(n)}\|_{L^q(\mathbb R^2)}\|\frac{1}{|z|^2}\mathbbm{1}_{B(0,\frac{d}{2}+a)\setminus B(0,\frac{d}{2}-a)}\|_{L^1(\mathbb R^2)}\\
\leq& C(q)\ln \left(\frac{d+2a}{d-2a}\right)d^{\frac{2}{q}}\Big(\sum_{\ell=1}^{N}|A^{(n)}_{\ell}|^q\Big)^{1/q}\\
\leq&C(q)\ln\left(\frac{1+2\eps_0}{1-2\eps_0}\right) d^{\frac 2q}\Big(\sum_{\ell=1}^{N}|A^{(n)}_{\ell}|^q\Big)^{1/q}.
\end{align*}
The last inequality is guaranteed by $a/d \leq \eps_0<1/2$.
Combining with (\ref{Jl}) we obtain that 
\begin{equation}\label{J-est}
\Big(\sum_{\ell=1}^{N}|J_{\ell}|^q \Big)^{1/q} 
\leq C\Big(\frac{a}{d}\Big)^{\frac{2}{q'}} \ln\left(\frac{1+2\eps_0}{1-2\eps_0}\right)\Big(\sum_{\ell=1}^N|A^{(n)}_{\ell}|^q\Big)^{1/q} .
\end{equation}

Putting together \eqref{A-2nd}, \eqref{I-est} and \eqref{J-est} in the two decompositions \eqref{A-decompo} and \eqref{decompo2} ends the proof of the lemma.
\end{proof}

Applying the previous lemma with $q=2$, we obtain that the method of reflections converges when $a/d < \varepsilon_0 \leq \varepsilon_{ref}$
where 
\begin{equation}\label{def-ref}
\varepsilon_{ref} : =\min \Big( \frac1{4 C_{ref}(2)}, \tilde\varepsilon_{ref}\Big) 
\end{equation}
with $\tilde\varepsilon_{ref}$ the unique solution in $(0,1/2)$ of the equation: 
\[
\tilde\varepsilon_{ref} \ln\left(\frac{1+2\tilde\eps_{ref}}{1-2\tilde\eps_{ref}}\right)= \dfrac{1}{4C_{ref}(2)}.
\]
Indeed, with this choice of $\varepsilon_{ref}$, it is clear that $\varepsilon_{0}<1/2$ and 
\[
\Big(\sum_{\ell=1}^N |A^{(n+1)}_{\ell}|^2\Big)^{1/2}\leq \frac12\Big(\sum_{\ell=1}^N |A^{(n)}_{\ell}|^2\Big)^{1/2} \leq \Big(\frac12\Big)^{n}\Big(\sum_{\ell=1}^N |A^{(1)}_{\ell}|^2\Big)^{1/2}.
\]
Let us recall \eqref{psi0-bd} which entails that (see the beginning of this section for notations $R_{f},M_{f}$):
\begin{equation}\label{est-A1}
\max_{\ell=1,\ldots,N} |A_{\ell}^{(1)}| = \max_{\ell=1,\ldots,N} |\nabla \psi_0(x_{\ell})| \leq C \|f\|^{1/2}_{L^{1}(\mathbb R^2)} \|f\|^{1/2}_{L^{\infty}(\mathbb R^2)}
\leq C R_{f}M_{f}.
\end{equation}
Even if in the previous argument, we used Lemma~\ref{Aapp} only for $q=2$, this lemma will be also used in Subsection~\ref{sect-stab} for $q\geq 2$ arbitrary large.

\medskip

The second step of the analysis is to obtain that the $(\psi^{(n)})_{n\in \mathbb N}$ yield good approximations of the exact solution $\psi_N.$ 
The proof of this result is based on two ingredients: a variational property of $\psi_N - \psi^{(n)}$ and a control of the second order expansion of $\psi^{(n)}$ on the $(\partial \K_{\ell}^a)_{\ell=1,\ldots,N}.$
This is the content of the next proposition. 

\begin{proposition}\label{1st term}
If $\eps_0 \leq \min(\eps_{ref},1/4)$, there exists a constant $C_{app}(R_{f},\eps_0)$ such that for any $n \geq 3,$ there holds:
\[
\|\psi^{(n)}-\psi_N\|_{\dot{H}^1(\mathcal{F}_N)}\leq C_{app}(R_{f},\eps_0)\left( \Big(\frac{a}{d}\Big)^{4} +a\right) M_{f}.
\]
\end{proposition}

\begin{proof}
Recalling the definitions \eqref{eq_defpsiN-2} of $\psi_N$ and \eqref{eq_phi_iter} of $\psi^{(n)}$ (see also the definition \eqref{reflection-a} of $V^a$ which verifies \eqref{Va-circ}), we note that $\psi^{(n)}-\psi_N$ belongs to $\dot H^1(\mathcal{F}_{N})$ (see the behavior at infinity \eqref{psi0-inf}, \eqref{psiN-inf} and Lemma~\ref{V1-decay}) and satisfies
\begin{equation}\label{psin-psiN}
 \left\{
 \begin{array}{ll}
 \Delta (\psi^{(n)}-\psi_N)=0, & \mathrm{in}~~\mathcal{F}_N,\\
 \psi^{(n)}-\psi_N=w^{(n)}, & \mathrm{on} ~~\partial \K_{\ell}^a,~~\forall\, \ell=1,\ldots,N,\\
 \displaystyle \int_{\partial \K_{\ell}^a} \partial_n (\psi^{(n)}-\psi_N) \d \sigma=0, & \forall\, \ell=1,\ldots,N,
 \end{array}
\right.
\end{equation}
where, for any $\ell= 1,\ldots,N$ and $x\in \partial \K_{\ell}^a$ we have defined
\begin{align*}
 w^{(n)}(x) :=& \psi_0(x) + \sum_{j=1}^{n}\sum_{\lambda =1}^N V^{a}[A^{(j)}_{\lambda}](x-x_\lambda) - \psi_{N,\ell}^*\\
= & \psi_0(x) + \sum_{j=1}^{n} \Big( A^{(j)}_{\ell}\cdot (x-x_{\ell})+ \sum_{\lambda \neq \ell} V^{a}[A^{(j)}_{\lambda}](x-x_\lambda) \Big) - \psi_{N,\ell}^*.
\end{align*}

The first argument of this proof is to notice that $\psi^{(n)}-\psi_N$ minimizes the $\dot H^1(\mathcal{F}_{N})$ on the set of $C^1$ functions which satisfy this boundary condition up to a constant. Namely, for any $w_{N}\in C^1_{c}(\overline{\mathcal{F}_{N}})$ which verifies
\begin{equation} \label{wN}
\partial_{\tau }(w_{N}-w^{(n)})=0 \quad \text{on } \partial\mathcal{F}_{N},
\end{equation}
we get by two integrations by parts and system \eqref{psin-psiN}
\begin{align*}
\int_{\mathcal F_N} |\nabla (\psi^{(n)} - \psi_N)|^2& = \int_{\partial\mathcal F_N} \partial_{n}((\psi^{(n)} - \psi_N)) w^{(n)}\d \sigma= \int_{\partial\mathcal F_N} \partial_{n}((\psi^{(n)} - \psi_N)) w_{N}\d \sigma\\
&=\int_{\mathcal F_N} (\nabla (\psi^{(n)} - \psi_N))\cdot \nabla w_{N},
\end{align*}
hence by the Cauchy–Schwarz inequality:
\begin{equation} \label{eq_clabase}
\int_{\mathcal F_N} |\nabla (\psi^{(n)} - \psi_N)|^2 \leq \int_{\mathcal F_N} |\nabla w_N|^{2}.
\end{equation}

Therefore, we create now a lifting (up to constants) $w_{N}$ of the boundary value $w^{(n)}$ and we estimate its $\dot H^1(\mathcal{F}_{N})$ norm. 
First, we define a cutoff function $\chi\in C_c^{\infty}(\R^2)$ such that 
\[
\chi\equiv1\text{ in }\K + B(0,1/2)\quad \text{and} \quad \chi\equiv0\text{ in }\R^2 \setminus (\K + B(0,1)).
\]
Second, we set for any $\ell=1,\dots,N$
\[
\tilde w_{N,\ell} (x):= \psi_0(x) + \sum_{j=1}^{n} \Big( A^{(j)}_{\ell}\cdot (x-x_{\ell})+ \sum_{\lambda \neq \ell} V^{a}[A^{(j)}_{\lambda}](x-x_\lambda) \Big) 
\]
and by $\hat w_{N,\ell}$ the mean value of $\tilde w_{N,\ell}$ on $(\K_{\ell}^a +B(0,a))\setminus \K_{\ell}^a$. We finally define
\[
w_{N}(x):=\sum_{\ell=1}^N (\tilde w_{N,\ell}(x) -\hat w_{N,\ell} )\chi\Big(\frac{x-x_{\ell}}{a}\Big),
\]
which clearly verifies \eqref{wN}. Therefore, by \eqref{eq_clabase}, our proof reduces now to estimate the $L^2$ norm of $\nabla w_{N}$, which decomposes as follows:
\begin{align*}
\|\nabla w_{N} \|_{L^2(\mathcal{F}_{N})}^2 
\leq& 2\Big\| \frac1a\sum_{\ell=1}^N (\tilde w_{N,\ell}(x) -\hat w_{N,\ell} ) (\nabla\chi)\Big(\frac{x-x_{\ell}}{a}\Big) \Big\|_{L^2(\mathcal{F}_{N})}^2 +2 \Big\| \sum_{\ell=1}^N \nabla \tilde w_{N,\ell}(x) \chi\Big(\frac{x-x_{\ell}}{a}\Big) \Big\|_{L^2(\mathcal{F}_{N})}^2 \\
\leq& \frac{2C_{\chi}}{a^2}\sum_{\ell=1}^N \int_{(\K_{\ell}^a +B(0,a))\setminus \K_{\ell}^a} |\tilde w_{N,\ell}(x) -\hat w_{N,\ell} |^2\d x
+ 2 \sum_{\ell=1}^N \int_{(\K_{\ell}^a +B(0,a))\setminus \K_{\ell}^a} | \nabla \tilde w_{N,\ell}(x) |^2\d x
\end{align*}
where we have used that $x \mapsto \chi((x-x_{\ell})/a)$ have disjoint supports. By a standard change of variable $y=(x-x_{\ell})/a$, a Poincar\'e-Wirtinger inequality on the domain $(\K +B(0,1))\setminus \K$ entails that
\[
\|\nabla w_{N} \|_{L^2(\mathcal{F}_{N})}^2 
\leq (2C_{\chi} C_{PW}+2) \sum_{\ell=1}^N \int_{(\K_{\ell}^a +B(0,a))\setminus \K_{\ell}^a} | \nabla \tilde w_{N,\ell}(x) |^2\d x.
\]
With the expression \eqref{A1} of $A_{\ell}^{(j)}$ and \eqref{eq_A_iter}, we compute
\begin{align*}
\nabla \tilde w_{N,\ell} (x)=& \nabla\psi_0(x) +\sum_{j=1}^{n} \Big( A^{(j)}_{\ell}+ \sum_{\lambda \neq \ell} \nabla V^{a}[A^{(j)}_{\lambda}](x-x_\lambda) \Big) \\
=&\nabla \psi_0(x)-\nabla \psi_{0}(x_{\ell} ) + \sum_{j=1}^{n-1} \sum_{\lambda \neq \ell} \Big( \nabla V^{a}[A^{(j)}_{\lambda}](x-x_\lambda)- \nabla V^{a}[A^{(j)}_{\lambda}](x_{\ell}-x_\lambda) \Big) \\
&+ \sum_{\lambda \neq \ell} \nabla V^{a}[A^{(n)}_{\lambda}](x-x_\lambda),
\end{align*}
hence
\begin{equation}\label{wN-decompo}
 \|\nabla w_{N} \|_{L^2(\mathcal{F}_{N})}^2 \leq 3(2C_{\chi} C_{PW}+2) (K_{1}+K_{2}^{(n)}+K_{3}^{(n)}) ,
\end{equation}
where
\begin{align*}
K_{1} &:= \sum_{\ell=1}^N \int_{(\K_{\ell}^a +B(0,a))\setminus \K_{\ell}^a} |\nabla \psi_0(x)-\nabla \psi_{0}(x_{\ell} ) |^2\d x,\\
K_{2}^{(n)} &:= \sum_{\ell=1}^N \int_{(\K_{\ell}^a +B(0,a))\setminus \K_{\ell}^a} \Big|\sum_{j=1}^{n-1} \sum_{\lambda \neq \ell} \nabla V^{a}[A^{(j)}_{\lambda}](x-x_\lambda)- \nabla V^{a}[A^{(j)}_{\lambda}](x_{\ell}-x_\lambda) \Big|^2\d x, \\
K_{3}^{(n)} &:= \sum_{\ell=1}^N \int_{(\K_{\ell}^a +B(0,a))\setminus \K_{\ell}^a} \Big| \sum_{\lambda \neq \ell} \nabla V^{a}[A^{(n)}_{\lambda}](x-x_\lambda) \Big|^2\d x.
\end{align*}

\medskip

We begin by $K_{3}^{(n)}$ because the analysis is almost the same as in the proof of Lemma~\ref{Aapp}. The explicit expansion of $V^1$ (see Lemma~\ref{V1-decay}) and the scaling law \eqref{scaling-V} give:
\begin{equation*}
\sum_{\lambda \neq \ell} \nabla V^{a}[A^{(n)}_{\lambda}](x-x_\lambda)
 =a^2\sum_{\lambda\neq \ell}\nabla \Big(\dfrac{(\M_{\K} A^{(n)}_{\lambda} )\cdot (x-x_{\lambda})}{|x-x_{\lambda}|^2}\Big) 
 + a^3 \sum_{\lambda\neq \ell} \frac{h_{1}(\frac{x-x_\lambda}a)}{|x-x_\lambda|^3}A^{(n)}_{\lambda}. 
\end{equation*}
For all $x \in \K_{\ell}^a +B(0,a)$ and $\lambda\neq \ell$, we use that $|x-x_{\ell}| \leq 3a \leq 3d/4 \leq 3 |x_{\ell}-x_{\lambda}|/4$ provided that $\varepsilon_{0}\leq 1/4$ to state that $|x-x_{\lambda}|\geq |x_{\ell}-x_{\lambda}|/4$, hence
\[
\sum_{\ell=1}^N \int_{(\K_{\ell}^a +B(0,a))\setminus \K_{\ell}^a} \Big| a^3 \sum_{\lambda\neq \ell} \frac{h_{1}(\frac{x-x_\lambda}a)}{|x-x_\lambda|^3} A^{(n)}_{\lambda} \Big|^2\d x \leq C a^8 \sum_{\ell=1}^N \Big( \sum_{\lambda\neq \ell} \frac{|A^{(n)}_{\lambda} | }{|x_{\ell} -x_\lambda|^3} \Big)^2 
\]
where we have used that $h_{1}$ is bounded. Like for the second sum of \eqref{A-decompo}, we state that the pseudo discrete Young's convolution inequality gives
\[
 \sum_{\ell=1}^N \int_{(\K_{\ell}^a +B(0,a))\setminus \K_{\ell}^a} \Big| a^3 \sum_{\lambda\neq \ell} \frac{h_{1}(\frac{x-x_\lambda}a)}{|x-x_\lambda|^3} A^{(n)}_{\lambda}\Big|^2\d x \leq C \frac{a^8}{d^6} \sum_{\ell=1}^N |A^{(n)}_{\ell}|^2.
\]

Concerning the other part of $K_{3}^{(n)}$, we notice that for any $x \in \K_{\ell}^a +B(0,a)$ and $\lambda\neq \ell$, $y \mapsto {a^2(\M_{\K}A^{(n)}_{\lambda})\cdot(x-y)}/{|x-y|^2}$ is harmonic in $B(x_\lambda,d/2)$. Hence by applying mean-value formula, we obtain that
\begin{align*}
a^2\sum_{\lambda\neq \ell}\nabla \Big(\dfrac{(\M_{\K} A^{(n)}_{\lambda} )\cdot (x-x_{\lambda})}{|x-x_{\lambda}|^2}\Big) 
&=\frac{4a^2}{\pi d^2}\sum_{\lambda\neq \ell}\int_{B(x_\lambda,d/2)}\nabla \Big( \dfrac{(\M_{\K}A^{(n)}_{\lambda} )\cdot (x-y)}{|x-y|^2}\Big) \d y \\
&=\frac{\pi a^2}{ d^2} \int_{\R^2\setminus B(x_{\ell},d/2)} \Big(D_{x}\dfrac{x-y}{|x-y|^2}\Big)^T \M_{\K} \Phi^{(n)}(y) \d y
\end{align*}
with $\Phi^{(n)}$ defined in \eqref{eq_Phi}. So we follow the proof of Lemma~\ref{Aapp} -- and keep the conventions \eqref{eq_Fn} and \eqref{eq_Gn} to define
the functions $F_n$ and $G_n$ -- to state that
\[
 \int_{(\K_{\ell}^a +B(0,a))\setminus \K_{\ell}^a} \Big| a^2\sum_{\lambda\neq \ell}\nabla \Big(\dfrac{(\M_{\K} A^{(n)}_{\lambda} )\cdot (x-x_{\lambda})}{|x-x_{\lambda}|^2}\Big) \Big|^2\d x \leq \tilde I_{\ell} + \tilde J_{\ell}
\]
where, on the one hand, we write
\begin{align*}
 \tilde I_{\ell} &= 2 \int_{(\K_{\ell}^a +B(0,a))\setminus \K_{\ell}^a} \Big| \frac{\pi a^2}{ d^2} F^{(n)}(x) \Big|^2\d x \\ 
 \sum_{\ell=1}^N \tilde I_{\ell} &\leq \frac{2\pi^2 a^4}{ d^4} \| F^{(n)} \|^2_{L^2}(\R^2) \leq C \frac{2\pi^2 a^4}{ d^2} \sum_{\ell=1}^N |A^{(n)}_{\ell}|^2;
\end{align*}
and, on the other hand, we compute
\begin{align*}
 \tilde J_{\ell} &\leq 2 \int_{(\K_{\ell}^a +B(0,a))\setminus \K_{\ell}^a} \Big( \frac{\pi a^2}{ d^2} G^{(n)}(x) \Big)^2\d x \\ 
 \sum_{\ell=1}^N \tilde J_{\ell}& \leq \frac{2\pi^2 a^4}{ d^4} \| G^{(n)} \|^2_{L^2}(\R^2) \leq C \frac{2\pi^2 a^4}{ d^2} \Big(\ln\Big(\frac{1+2\eps_0}{1-2\eps_0}\Big)\Big)^2 \sum_{\ell=1}^N |A^{(n)}_{\ell}|^2.
\end{align*}
Putting together the previous estimates, we have proved that
\[
 K_{3}^{(n)} \leq C\frac{ a^4}{ d^2} \Big(1+\ln\Big(\frac{1+2\eps_0}{1-2\eps_0}\Big)\Big)^2 \sum_{\ell=1}^N |A^{(n)}_{\ell}|^2.
\]
Applying Lemma~\ref{Aapp} and recalling our choice \eqref{def-ref} of $\varepsilon_{ref}$, we conclude by \eqref{est-A1} that
\begin{align}
 K_{3}^{(n)} 
 &\leq 
 C \frac{ a^4}{ d^2} \Big(1+\ln\Big(\frac{1+2\eps_0}{1-2\eps_0}\Big)\Big)^2
 \Bigg[ C_{ref}(2)\frac{a}{d} \Big(1+\ln\Big( \frac{1+2\eps_0}{1-2\eps_0}\Big)\Big) \Bigg]^{2(n-1)} R_{f}^2M_{f}^2 N \nonumber\\
&\leq C \frac{ a^8}{ d^6} \Big(1+\ln\Big(\frac{1+2\eps_0}{1-2\eps_0}\Big)\Big)^6
 \Bigg[ C_{ref}(2)\varepsilon_{ref} \Big(1+\ln\Big( \frac{1+2\varepsilon_{ref}}{1-2\varepsilon_{ref}}\Big)\Big) \Bigg]^{2n-6} R_{f}^2M_{f}^2 N \nonumber \\
&\leq C \Big(\frac{ a}{ d}\Big)^8 \Big(1+\ln\Big(\frac{1+2\eps_0}{1-2\eps_0}\Big)\Big)^6
 R_{f}^2M_{f}^2 , \label{est-K3}
\end{align}
for any $n\geq 3$, where we have used that our assumption\eqref{main_assumption} on $\mathcal{F}_{N}$ entails $N\leq C(K_{PM}) d^{-2}$. Actually, we note that $K_{3}^{(n)}$ could be smaller if necessary because we could extract additional power of $(a/d)$ in the previous argument.

\medskip 

Concerning $K_{1}$, we simply use the log-lipschitz estimate of $\nabla \psi_{0}$ \eqref{psi0-lips} to write for any $x\in (\K_{\ell}^a +B(0,a))\setminus \K_{\ell}^a$
\[
|\nabla\psi_0(x)-\nabla\psi_0(x_{\ell})| \leq C (1+ R_{f}^2)M_{f} |x-x_{\ell}| | \ln|x-x_{\ell}| | \leq C (1+ R_{f}^2)M_{f} |x-x_{\ell}|^{3/4} \leq C (1+ R_{f}^2)M_{f} a^{3/4},
\]
hence
\begin{equation}\label{est-K1}
K_{1} \leq C (1+ R_{f}^2)^2M_{f}^2 a^{3/2} a^2 N\leq C (1+ R_{f}^2)^2M_{f}^2 a^{3/2} \frac{a^2}{d^2} \leq C (1+ R_{f}^2)^2M_{f}^2 \Big( a^{2} + \Big( \frac{a}{d}\Big)^8 \Big).
\end{equation}

Using Lemma~\ref{V1-decay} with $m_{1}+m_{2}=2$ and the scaling law \eqref{scaling-V}, we consider now $K_{2}^{(n)}$: we have for any $x\in(\K_{\ell}^a +B(0,a))\setminus \K_{\ell}^a$
\begin{align*}
 \Big|\sum_{j=1}^{n-1} \sum_{\lambda \neq \ell} \nabla V^{a}[A^{(j)}_{\lambda}](x-x_\lambda)- \nabla V^{a}[A^{(j)}_{\lambda}](x_{\ell}-x_\lambda) \Big|
 & \leq \sum_{j=1}^{n-1} \sum_{\lambda \neq \ell} 3a \max_{\K_{\ell}^a +B(0,a)}|\nabla^2 V^{a}[A_{\lambda}^{j}](x-x_{\lambda})| 
 \\
& \leq
Ca^3\sum_{j=1}^{n-1}\sum_{\lambda \neq \ell}\frac{|A^{(j)}_{\lambda}|}{|x_{\ell}-x_\lambda|^3},
\end{align*}
where we have again used that $|x-x_{\lambda}|\geq |x_{\ell}-x_{\lambda}|/4$ for all $x \in \K_{\ell}^a +B(0,a)$ and $\lambda\neq \ell$ (provided that $\varepsilon_{0}\leq 1/4$). As in the beginning of the proof of Lemma~\ref{Aapp}, we apply the pseudo discrete Young's convolution inequality for $b_{\lambda} = \sum_{j=1}^{n-1} |A^{(n)}_{\lambda}|$ and $a_{\ell \lambda }=|x_{\ell}-x_\lambda|^{-3}$ for $\ell \neq \lambda$ (otherwise $a_{\ell \ell}=0$) to get
\begin{align*}
(K_{2}^{(n)})^{1/2} &\leq Ca^4 \Bigg( \sum_{\ell=1}^N \Big(\sum_{\lambda \neq \ell}\frac{\sum_{j=1}^{n-1} |A^{(j)}_{\lambda}|}{|x_{\ell}-x_\lambda|^3} \Big)^2 \Bigg)^{1/2} 
\leq Ca^4 \sup_{\ell}\sum_{\lambda\neq \ell} \frac{1 }{|x_{\ell}-x_\lambda|^3} \Big(\sum_{\ell=1}^N \Big(\sum_{j=1}^{n-1} |A^{(j)}_{\lambda}|\Big)^2\Big)^{1/2} \\
&\leq Ca^4 \Big(C d^{-2} \int_{B(0,d/2)^c} |x|^{-3}\d x \Big)\sum_{j=1}^{n-1} \Big(\sum_{\ell=1}^N |A^{(j)}_{\ell}|^2\Big)^{1/2} 
\leq C a\Big(\frac ad \Big)^{3} \sum_{j=1}^{n-1} \frac1{2^{j-1}} \Big(\sum_{\ell=1}^N |A^{(1)}_{\ell}|^2\Big)^{1/2},
\end{align*}
where we have used Lemma~\ref{Aapp} together with the definition \eqref{def-ref} of $\varepsilon_{ref}$. Therefore, we conclude by \eqref{est-A1} and that $N\leq C(K_{PM}) d^{-2}$ that
\begin{equation}\label{est-K2}
 K_{2}^{(n)} \leq C \Big(\frac ad \Big)^{8} R_{f}^2 M_{f}^2 .
\end{equation}

\medskip

We are now able to conclude this proof: using the variational property \eqref{eq_clabase} with the decomposition \eqref{wN-decompo}, then the estimates \eqref{est-K3}-\eqref{est-K2} gives for all $n\geq 3$ and $\varepsilon_{0}\leq \min(\eps_{ref},1/4)$:
\[
\int_{\mathcal F_N} |\nabla (\psi^{(n)} - \psi_N)|^2 \leq C (1+ R_{f}^2)^2 \Big(1+\ln\Big(\frac{1+2\eps_0}{1-2\eps_0}\Big)\Big)^6
 M_{f}^2 \Big( a^{2} + \Big( \frac{a}{d}\Big)^8 \Big)
\]
where $C$ depends only on $K_{PM}$ and $\K$.
\end{proof}

 In the above lemma, a part of the error between $\psi_N$ and its approximation $\psi^{(n)}$ is measured by the size $a$ of the holes. 
To be able to proceed with our method we thus need to show that the size of the holes is small when the distribution of the holes is approximated by a continuous volume fraction. So, we provide in the next lemma a control of the radius size $a$ of the holes with respect to the distance in $W^{-1,p}(\mathbb R^2)$ between the indicator function $\mu$ of $\mathcal{F}_{N}$ (defined in \eqref{def-mu}) and the limit volume fraction $k$.

\begin{lemma} \label{lem_def_eps_s}
If $\varepsilon_0 \leq 1/2$ and $p \in (1,\infty)$, there exists a constant $C(p,\varepsilon_0)$ depending only on $p$ and $\varepsilon_0$ such that
\[
a \leq C(p,\varepsilon_0) \|\mu - k \|^{\frac{p}{p+2}}_{W^{-1,p}(\mathbb R^2)}.
\]
\end{lemma}

This lemma gives obviously the following corollary of Proposition~\ref{1st term}.

\begin{corollary}\label{1st term-bis}
If $\eps_0 \leq \min(\eps_{ref},1/4)$, given $p \in (1,\infty)$, there exists a constant $C_{app}(R_{f},p,\eps_0)$ such that for any $n \geq 3,$ there holds:
\[
\|\psi^{(n)}-\psi_N\|_{\dot{H}^1(\mathcal{F}_N)}\leq C_{app}(R_{f},p,\eps_0) M_{f}\left( \Big(\frac{a}{d}\Big)^{4} +\|\mu - k\|^{\frac{p}{p+2}}_{W^{-1,p}(\mathbb R^2)}\right).
\]
\end{corollary}

\begin{proof}[Proof of Lemma~\ref{lem_def_eps_s}]
Let denote 
\[
\delta =\dfrac{1}{2{\varepsilon_0}} a. 
\]
We remark that, with assumption \eqref{main_assumption}, the open set $B(x_{\ell},\delta)$ does not intersect the other $B(x_{\lambda},\delta)$ ($\lambda\neq \ell$) if $\varepsilon_{0}\leq 1/2$. 
We introduce then $\chi$ a plateau-function such that 
\[
\mathds{1}_{B(0,a)} \leq \chi \leq \mathds{1}_{B(0,\delta)}
\qquad
\text{with} 
\qquad
|\nabla \chi| \leq C_{\chi} \dfrac{\mathds{1}_{B(0,\delta)} - \mathds{1}_{B(0,a)}}{\delta},
\]
since $\delta > a.$ 
Given a center of hole $x_{\ell}$ ($\ell=1,\ldots,N$) we have, on the one hand (since $k \in \mathcal{FV}(\varepsilon_0)$):
\begin{align*}
 \langle \mu -k, \chi(\cdot-x_{\ell}) \rangle &= \int_{B(x_{\ell},\delta)} (\mathds{1}_{B(x_{\ell},a)}(x) - k(x))\chi(x-x_{\ell})\d x \\
& \geq a^2\pi - \|k\|_{L^{\infty}} \delta^2 \pi \geq a^2\pi \Big( 1 - \varepsilon_{0}^2(1/(2\varepsilon_0))^2\Big)= \frac{3\pi a^2}4.
\end{align*}
On the other hand, we also have the bound:
\begin{align*}
\left| \langle \mu -k , \chi(\cdot-x_{\ell}) \rangle \right| 
&\leq \|\mu - k \|_{W^{-1,p}(\mathbb R^2)} \|\chi\|_{W^{1,p'}(\mathbb R^d)} \leq C(p) \|\mu - k \|_{W^{-1,p}(\mathbb R^2)} \dfrac{\delta^{\frac 2{p'}}}{\delta} \\
&\leq C(p,\varepsilon_{0}) \|\mu - k \|_{W^{-1,p}(\mathbb R^2)}a^{\frac 2{p'} -1}
\end{align*}
which ends the proof.
\end{proof}

We note here that we do not get a better estimate when $n$ becomes larger. This is due to the fact that we only correct the first order on the boundaries in our method of reflections. We could have better estimates by correcting further orders in the boundary conditions. As a consequence, we shall stick to the case $n=3$ below letting $N \to \infty$. Despite we only look at the first terms in the sequence $(\psi^{(n)})$, we provided a convergence result for the whole sequence. Our motivation here is twofold. First, we want to point out that the sequence converges to something which is not the exact solution $\psi_N$. Second, the argument ensuring that we get a more precise approximation to $\psi_N$ with $\psi^{(3)}$ than with $\psi^{(0)}$, independantly of the number $N$ of obstacles, is worth convergence of the whole sequence.

\subsection{Construction of $\psi_c$ and first order expansion} \label{sec_expansion_Psic}

For any $M_{\K}\in \mathcal{M}_{2}(\R)$, we continue this section with an existence theory for the elliptic problem \eqref{def-psic} that we recall here:
\begin{equation} \label{def-psic-2}
\left\{
\begin{aligned}
{\rm div}[({\rm I}_{2}+kM_{\K}) \nabla \psi_c] &= f \quad \text{ in $\mathbb R^2$}\,, \\
\lim_{|x| \to \infty} \nabla \psi_c(x) &= 0 .
\end{aligned}
\right.
\end{equation}
We recall that $f \in L^{\infty}(\mathbb R^2)$ and $k \in \mathcal {FV}(\varepsilon_0)$ have compact supports. 
We state first the well-posedness of \eqref{def-psic-2} and give a first estimate on $\psi_{c}$ and on $\psi_{c}-\psi_{0}$ (recalling that $\psi_{0}=\Delta^{-1} f$ is defined in \eqref{form-psi0}).

\begin{proposition}\label{uc-existence}
Let $q \in (2,\infty).$ There exists a constant $\eps_c(q)>0$ depending on $q$ and $M_{\K}$ such that, if $\varepsilon_0 \leq \varepsilon_c(q)$ there exists, for any $f\in L^\infty_{c}(\R^2)$, a unique (up to a constant) solution $\psi_c\in \dot{W}^{1,q}(\R^2)$ to (\ref{def-psic}). Moreover, there exists $C(q)$ independent of $f$ such that
\[
\| \psi_c \|_{\dot{W}^{1,q}(\R^2)} \leq C(q) \| f\|_{L^1\cap L^\infty(\R^2)}
\quad\text{and}\quad 
\|\psi_{c}-\psi_0\|_{\dot{W}^{1,q}(\mathbb R^2)}\leq C(q) \|k\|_{L^{\infty}(\mathbb R^2)} \| f\|_{L^1\cap L^\infty(\R^2)}.
\]
\end{proposition}
\begin{proof}
Let $f \in L^{\infty}(\mathbb R^2)$ with compact support and $q>2.$ We prove this statement by a perturbative method. 

For this, we start by noting that $\Delta^{-1}$ (as defined in \eqref{form-psi0}) is a bounded operator from $L^p(\R^2)$ to $\dot{W}^{1,q}(\mathbb R^2)$ where $p=2q/(q+2)$. Indeed, the operator $\nabla \Delta^{-1}$ is associated with the kernel $y \mapsto y/(2\pi|y|^2) \in L^{2,\infty}(\mathbb R^2)$, so for any $g\in L^p(\R^2)$, the Hardy-Littlewood-Sobolev Theorem (see e.g. \cite[Theo. V.1]{Stein} with $\alpha=1$) where $\frac1{p}=\frac1{q}+\frac12$ gives:
\begin{equation}\label{est-Delta}
\|\Delta^{-1}g\|_{\dot{W}^{1,q}(\R^2)}
= \|\nabla\Delta^{-1}g \|_{L^{q}(\mathbb R^2)} 
 \leq C(q) \|g\|_{L^p(\mathbb R^2)}.
\end{equation}
This inequality holds for $q\in (2,\infty)$, and $p \in (1,2)$.
Moreover, by the Calder\'on-Zygmund inequality, we also know the following continuity result:
\begin{equation}\label{est-Delta2}
\|\Delta^{-1} \div g\|_{\dot{W}^{1,q}(\R^2)}
= \|\Delta^{-1} \nabla \div g \|_{L^{q}(\mathbb R^2)} 
 \leq C(q) \|g\|_{L^q(\mathbb R^2)}.
\end{equation}

Next, we remark that $\psi_c$ is a solution to \eqref{def-psic-2} if 
\[
\Delta \psi_c = f - \div(kM_{\K} \cdot \nabla \psi_c). 
\]
We set: 
\[
\psi_{c,0}:=\psi_{0}=\Delta^{-1} f,
\]
and, for arbitrary $n\geq 1$:
\[
\psi_{c,n}:=\Delta^{-1} f- \mathcal{L} \psi_{c,n-1},
\]
where 
\[ 
\mathcal{L} \psi:=\Delta^{-1}\div(kM_{\K} \nabla \psi) .
\]

By \eqref{est-Delta}, we state that $\psi_{c,0}$ belongs to $\dot{W}^{1,q}(\R^2)$ and
\[
\|\psi_{c,0} \|_{\dot{W}^{1,q}(\R^2)} \leq C(q) \| f\|_{L^1(\R^2)}^{\frac1q} \| f\|_{L^\infty(\R^2)}^{1-\frac1q} \leq C(q)\| f\|_{L^1\cap L^\infty(\R^2)}.
\]
Next, we note that $\mathcal{L}$ is a linear operator from $\dot{W}^{1,q}(\R^2)$ to itself such that
\begin{align*}
 \|\mathcal{L}\psi\|_{\dot{W}^{1,q}(\R^2)}& \leq C(q)\|kM_{\K} \nabla \psi \|_{L^q(\R^2)} \\
&\leq C(q,M_{\K}) \| k\|_{L^\infty} \|\nabla \psi\|_{L^{q}(\mathbb R^2)} \leq C(q,M_{\K}) \eps(q)^2 \| \psi\|_{\dot{W}^{1,q}(\R^2)},
\end{align*}
where we used \eqref{est-Delta2} and $k \in \mathcal FV(\varepsilon(q)).$ By choosing $\eps(q) = 1/\sqrt{2C(q,M_{\K})}$, we obtain that $(\psi_{c,n})$ is a Cauchy sequence in $\dot{W}^{1,q}(\R^2)$ which converges to $\psi_{c}$, solution of \eqref{def-psic-2}. This concludes the existence of $\psi_c$ for each fixed $2<q<\infty$, and the uniqueness comes directly from the fact that $\| \mathcal{L}\| \leq 1/2$.

By this Banach fixed point argument, we also have
\[
\| \psi_{c}-\psi_{0} \|_{\dot{W}^{1,q}(\R^2)} \leq \sum_{n=0}^\infty \| \psi_{c,n+1}-\psi_{c,n} \|_{\dot{W}^{1,q}(\R^2)} \leq 2 \| \psi_{c,1}-\psi_{c,0} \|_{\dot{W}^{1,q}(\R^2)}= 2 \| \mathcal{L}\psi_{0} \|_{\dot{W}^{1,q}(\R^2)}
\]
hence
\[
\| \psi_{c}-\psi_{0} \|_{\dot{W}^{1,q}(\R^2)} \leq 2 C(q,M_{\K}) \| k\|_{L^\infty} C(q) \| f\|_{L^1\cap L^\infty(\R^2)}
\]
and 
\[
\| \psi_{c} \|_{\dot{W}^{1,q}(\R^2)} \leq (1+ 2 C(q,M_{\K}) \| k\|_{L^\infty} )C(q) \| f\|_{L^1\cap L^\infty(\R^2)}.
\]
Therefore the proposition is proved.
\end{proof}

\begin{remark}\label{uc-expansion}
As a direct consequence to the previous proof is that
\[
\tilde{\psi}_c : =\psi_{c,1} = \psi_{0} - \Delta^{-1}\div(kM_{\K} \nabla \psi_{0})
\]
satisfies
\[
\| \psi_{c}-\tilde{\psi}_c \|_{\dot{W}^{1,q}(\R^2)} \leq \sum_{n=1}^\infty \| \psi_{c,n+1}-\psi_{c,n} \|_{\dot{W}^{1,q}(\R^2)} \leq 2 \| \psi_{c,2}-\psi_{c,1} \|_{\dot{W}^{1,q}(\R^2)}= 2 \| \mathcal{L}( \psi_{c,1}-\psi_{c,0} )\|_{\dot{W}^{1,q}(\R^2)}
\]
hence
\[
\| \psi_{c}-\tilde{\psi}_c \|_{\dot{W}^{1,q}(\R^2)} \leq 2 \| \mathcal{L} \|^2 \| \psi_{0} \|_{\dot{W}^{1,q}(\R^2)}\leq 2 C(q,M_{\K})^2 \| k\|_{L^\infty}^2 C(q) R_{f}^{2/q} M_{f}.
\]
This gives that $\tilde{\psi}_c$ is the first order expansion of $\psi_c$ w.r.t. the parameter $k$:
\[
\| \psi_{c}-\tilde{\psi}_c \|_{\dot{W}^{1,q}(\R^2)} \leq C(q,R_{f})M_{f} \|k\|^2_{L^{\infty}(\mathbb R^2)}
\]
for any $\varepsilon_{0}\leq \eps_c(q)$.
\end{remark}

\subsection{Stability estimate}\label{sect-stab}
In the above paragraph, we constructed a family $(\psi^{(n)})_{n\in \mathbb N}$ of approximations of $\psi_N.$ 
It turns out that the error is not improved by taking $n \geq 3.$ This is related to the fact that we correct only the first-order expansion of $\psi^{(n)}$ on $\partial \K_{\ell}^a$ in the recursive process. So, we restrict to index $n=3$ in what follows and we denote $\psi^{(3)}$ by $\bar{\psi}_N$.

\medskip

In this part, we show that, if $\varepsilon_0$ is small enough, the leading term of $\bar{\psi}_N$ when $N$ is large (meaning that $\mu$ is close to $k$) is given by $\tilde{\psi}_c$ (defined in the previous subsection) with the definition
$M_{\K} := 2 \M_{\K},$
where $\M_{\K}$ is defined in Lemma~\ref{V1-decay}. 
By Remark~\ref{rem-disk}, we can notice in the case of the unit disk that $M_{\K}=2 {\rm I}_{2}$.

\medskip

For this, we introduce the two following functions on $\R^2$:
\[
\phi_N(x):=-\sum_{\ell=1}^N\int_{B(x_{\ell},a)} (M_{\K}\nabla\psi_0(y))\cdot\frac{x-y}{2\pi|x-y|^2}\d y,~~\phi:=-{\rm div }\Delta^{-1}(k M_{\K}\nabla\psi_0).
\]
In particular, we remark that we have then $\tilde{\psi}_c:=\psi_0+\phi.$
The main result of this part is the following proposition:

\begin{proposition}\label{second term} 
Let $p\in(1,\infty)$ and $\eta \in (0,1)$ given. If $\eps_0 \leq 1/4$, for any compact subset $\mathcal{O}$ of $\R^2$, there exists a constant $C(\mathcal{O},R_{f},p,\eps_0,\eta)$ such that:
\[
\|\bar{\psi}_N-\tilde{\psi}_c\|_{L^2(\mathcal{O}\cap \mathcal{F}_{N})}\leq C(\mathcal O, R_{f},p,\eps_0,\eta) M_{f} \left[ \left( \dfrac{a}{d}\right)^{3-\eta} +
\|\mu -k\|^{ \frac{p(1-\eta)}{p+2}}_{W^{-1,p}(\mathbb R^2)}
 + \|\mu -k\|^{\frac12}_{W^{-1,p}(\mathbb R^2)} \right].
\]
\end{proposition}

\begin{proof}
We fix $p \in (1,\infty)$ and $\mathcal O \Subset \mathbb R^2$ for the whole proof. 
According to the definitions of $\bar{\psi}_N$ \eqref{eq_phi_iter} and $\phi$, we notice that 
\[
\bar{\psi}_N-\tilde{\psi}_c=\phi^{(1)}+\phi^{(2)}+\phi^{(3)}-\phi=r_{1}+r_{2}+r_{3},
\]
where
\[
r_{1}:=\phi_N-\phi,~~r_{2}:=\phi^{(1)}-\phi_N~~\mathrm{and}~~r_{3}:=\phi^{(2)}+\phi^{(3)}.
\]
We first notice that 
\[
r_{1}(x)=\int_{\R^2}(k(y)-\mu(y)) (M_{\K}\nabla\psi_0(y))\cdot\frac{x-y}{2\pi|x-y|^2}\d y, \quad \forall \, x \in \R^2,
\]
which, combined twice with \eqref{psi0-bd} and the fact that $k -\mu$ is supported in $K_{PM}$, gives 
\[
 \| r_{1} \|_{L^\infty(\R^2)}
 \leq C \| (k -\mu)\nabla\psi_0 \|_{L^1}^{1/2} \| (k -\mu)\nabla\psi_0 \|_{L^\infty}^{1/2}
 \leq C \| (k -\mu) \chi \nabla\psi_0 \|_{L^1}^{1/2} R_{f}^{1/2}M_{f}^{1/2} 
\]
where $\chi\in C^\infty_{c}(\R^2)$ such that $\chi\equiv 1$ on $K_{PM}$. Hence, we have
\[
 \| r_{1} \|_{L^\infty(\R^2)}
 \leq C \| k -\mu \|_{ W^{-1,p}(\mathbb R^2)}^{1/2} \| \chi \nabla\psi_0 \|_{ W^{1,p'}(\mathbb R^2)}^{1/2} R_{f}^{1/2}M_{f}^{1/2} 
 \leq C(p,\chi, R_{f})\| k -\mu \|_{ W^{-1,p}(\mathbb R^2)}^{1/2} M_{f},
\]
where we have again used \eqref{psi0-bd} together with the Calder\'on-Zygmund inequality (in order to state that $\|D^2 \psi_{0}\|_{L^{p'}(\mathbb R^2)}\leq C(p')\| f \|_{L^{p'}(\mathbb R^2)}$). This ends the estimate for $r_{1}$:
\begin{equation}\label{r1N}
 \| r_{1} \|_{L^2(\mathcal O\cap \mathcal{F}_{N})} \leq |\mathcal{O}|^{1/2}C(p,K_{PM}, R_{f})\| k -\mu \|_{ W^{-1,p}(\mathbb R^2)}^{1/2} M_{f}.
\end{equation}

\medskip 

To compute $\|r_{2}\|_{L^2(\mathcal O\cap \mathcal F_N)}$ we split:
\begin{equation} \label{eq_r2_split}
\|r_{2}\|^2_{L^2(\mathcal O\cap \mathcal F_N)} 
\leq \|r_2\|^2_{L^2(\mathcal O \setminus \bigcup B(x_{\ell},2a))} + 
\sum_{\ell=1}^N \|r_{2}\|^2_{L^2(B(x_{\ell},2a))} .
\end{equation}
For the first term, we notice from the expansion of $V^1$ (see Lemma~\ref{V1-decay}) and the scaling law \eqref{scaling-V}, there holds, for $x\in \mathcal O \setminus \bigcup B(x_{\ell},2a)$,
\begin{align*}
 r_{2}(x)=&\int_{\R^2}\big(M_{\K} \nabla\psi_0(y)\big)\mu(y)\cdot\frac{x-y}{2\pi|x-y|^2}\d y -\sum_{\ell=1}^N V^{a}[\nabla\psi_0(x_{\ell})](x-x_\ell) \\
=&\int_{\R^2}\big(M_{\K} \nabla\psi_0(y)\big)\mu(y)\cdot\frac{x-y}{2\pi|x-y|^2}\d y \\
&-\sum_{\ell=1}^N a^2\dfrac{(\M_{\K}\nabla\psi_0(x_{\ell}))\cdot (x-x_\ell)}{|x-x_\ell|^2} + a^3\nabla\psi_0(x_{\ell})\cdot \frac{h_{0}(\frac{x-x_\ell}a)}{|x-x_\ell|^{2}}.
\end{align*}
As often in this section, we use the harmonicity of the function $y\mapsto (x-y)/|x-y|^2$ on $B(x_{\ell},a)$ since $\dist(x,\{x_1,\ldots,x_N\})\geq 2a$ in this first case.
This yields that, for $x \in \mathbb R^2 \setminus \bigcup B(x_{\ell},2a):$
\begin{align*}
 r_{2}(x)
=&\int_{\R^2}\big( M_{\K}\nabla\psi_0(y)\big)\mu(y)\cdot\frac{x-y}{2\pi|x-y|^2}\d y - \sum_{\ell=1}^N \frac1{\pi}\int_{B(x_{\ell},a)} \dfrac{(\M_{\K}\nabla\psi_0(x_{\ell}))\cdot (x-y)}{|x-y|^2}\d y \\
&-\sum_{\ell=1}^N a^3\nabla\psi_0(x_{\ell})\cdot \frac{h_{0}(\frac{x-x_\ell}a)}{|x-x_\ell|^{2}} \\
=& \frac1{2\pi}\sum_{\ell=1}^N \int_{B(x_{\ell},a)}\Big(M_{\K}( \nabla\psi_0(y)-\nabla\psi_0(x_{\ell}))\Big)\cdot\dfrac{ x-y}{|x-y|^2}\d y
- \sum_{\ell=1}^N a^3\nabla\psi_0(x_{\ell})\cdot \frac{h_{0}(\frac{x-x_\ell}a)}{|x-x_\ell|^{2}}. 
\end{align*}
Let denote by $r_{2}^m(x)$ and $r_{2}^{r}(x)$ respectively the two terms on the right-hand side of this last equality. By \eqref{psi0-lips}, we get for any $y\in B(x_{\ell},a)$:
\[
\big| \nabla\psi_0(y)-\nabla\psi_0(x_{\ell})\big|\leq C(1+R_{f}^2)M_{f} a |\ln a|\leq C(1+R_{f}^2)M_{f} a^{3/4},
\]
hence
\begin{align*}
 \| r_{2}^{m} \|_{L^\infty(\mathcal O \setminus \bigcup B(x_{\ell},2a))}
 \leq &C(R_{f})M_{f} a^{3/4} \int_{\R^2} \frac{\mu (y)}{|x-y|}\, dy\\
 \leq &C(R_{f})M_{f} a^{3/4}\| \mu \|_{L^1}^{1/2} \| \mu \|_{L^\infty}^{1/2}\\
 \leq &C(R_{f})M_{f} a^{3/4} (a^2 N)^{1/2},
\end{align*}
where we have used a slightly stronger version of \eqref{psi0-bd} (see \cite[App. 2.3]{MarPul}). Using again that $N\leq C(K_{PM})d^{-2}$, we finally obtain
\[ \| r_{2}^{m} \|_{L^\infty(\mathcal O \setminus \bigcup B(x_{\ell},2a))}
 \leq C(K_{PM},R_{f})M_{f} \Big(a + \Big(\frac ad\Big)^{4}\Big).
\]

As for the remainder term $r_{2}^{r}$ in the expansion of $r_{2}$, we use that $h_{0}$ is bounded -- and that $B(x_{\ell},a) \cap B(x,a) = \emptyset$ for arbitrary 
$\ell \in \{1,\ldots,N\}$ (since $x \in \mathbb R^2 \setminus \bigcup B(x_\ell,2a)$) -- to state that:
\begin{align*}
| \sum_{\ell=1}^N a^3\nabla\psi_0(x_{\ell})\cdot \frac{h_{0}(\frac{x-x_\ell}a)}{|x-x_\ell|^{2}} \Big| 
\leq& C a \| \nabla \psi_{0} \|_{L^\infty} \int_{K_{PM} \setminus B(x,a)} \frac{1}{|x-y|^2} \d y \\
\leq& C(\mathcal O, K_{PM}, R_{f}) M_{f} a\Big( |\ln a| + 1 \Big).
\end{align*}
Considering that there exists $C_{\eta} >0 $ for which $a |\ln a|\leq C_{\eta}a^{1-\eta}$ whatever the value of $\eta \in (0,1)$, we finally get by Lemma~\ref{lem_def_eps_s} that:
\begin{equation} \label{eq_r2-partie1}
 \| r_{2} \|_{L^{\infty}(\mathcal O \setminus \bigcup B(x_{\ell},2a))}
 \leq C(p,\eta,\varepsilon_0,\mathcal O,K_{PM},R_{f}) M_{f} \Big( \Big(\frac ad\Big)^{4} + \|\mu - k \|^{\frac{p(1-\eta)}{p+2}}_{W^{-1,p}(\mathbb R^2)} \Big)
 \end{equation}
provided $\varepsilon_{0}\leq 1/4$. 

\medskip

To bound $r_2$ it remains to compute an upper bound for 
\[
\sum_{\ell=1}^{N} \|r_2\|^2_{L^2(B(x_\ell,2a))}.
\]
For this, given $\ell \in \{1,\ldots,N\}$ we write again $r_2(x) = r^{m}_{\ell}(x) + r^r_{\ell}(x)$ with:
\begin{align*}
r^m_{\ell}(x) & = 
\int_{\R^2 \setminus B(x_{\ell},a)} \big(M_{\K} \nabla\psi_0(y)\big)\mu(y)\cdot\frac{x-y}{2\pi|x-y|^2}\d y - \sum_{\lambda \neq \ell} V^{a}[\nabla\psi_0(x_{\lambda})](x-x_{\lambda}) \,, 
\\
r^r_{\ell}(x) & = \int_{B(x_{\ell},a)} \big(M_{\K} \nabla\psi_0(y)\big)\mu(y)\cdot\frac{x-y}{2\pi|x-y|^2}\d y - V^{a}[\nabla\psi_0(x_{\ell})](x-x_\ell). \\
\end{align*}
We control $r^m_{\ell}$ as the previous term $r_2^m.$ First, we remark that, on $B(x_{\ell},2a)$ there holds:
 \[ r^m_{\ell}(x) = \frac1{2\pi}\sum_{\lambda \neq \ell} \int_{B(x_{\lambda},a)}\Big(M_{\K}( \nabla\psi_0(y)-\nabla\psi_0(x_{\lambda}))\Big)\cdot\dfrac{ x-y}{|x-y|^2}\d y
- \sum_{\lambda \neq \ell} a^3\nabla\psi_0(x_{\ell})\cdot \frac{h_{0}(\frac{x-x_\lambda}a)}{|x-x_\lambda|^{2}}
\] 
that we bound similarly as the previous $r_{2}^m.$ This yields:
\[
|r_{\ell}^m(x)| \leq C(K_{PM},R_f) M_f \left( \left( \dfrac{a}{d}\right)^4 + \|\mu - k \|^{\frac{p(1-\eta)}{p+2}}_{W^{-1,p}(\mathbb R^2)}\right) \quad \forall \, x \in B(x_\ell,2a). 
\]
We note that this bound is not optimal since we merely used that the distance between two centers $x_{\lambda}$ is larger than $a$ (while there is a distance larger than $d$).
As for the second term, we have, for $x \in B(x_{\ell},2a),$ by the scaling property \eqref{scaling-V}:
\[
 |V^{a}[\nabla \psi_0(x_{\ell})](x-x_{\ell})| \leq a |\nabla \psi_0(x_{\ell}) | ( \| V^1[e_{1}]\|_{L^\infty}+ \| V^1[e_{2}]\|_{L^\infty}) \leq C(\K)R_{f} M_{f}a
 \]
 and
\[
\left| \int_{B(x_{\ell},a)} \big(M_{\K} \nabla\psi_0(y)\big) \cdot\frac{x-y}{2\pi|x-y|^2}\d y \right| \leq C(R_f,\K) M_f a,
\]
so that $|r^r_{\ell}(x)| \leq C(R_f,\K) M_f a$ on $B(x_{\ell},2a).$ Hence, recalling Lemma~\ref{lem_def_eps_s}, we have finally:
\begin{equation} \label{eq_r2-partie-2}
\sum_{\ell=1}^{N} \|r_2\|^2_{L^2(B(x_\ell,2a))} \leq C(K_{PM}) M_f^2 \left(\dfrac{a}{d} \right)^2 \left( \|\mu - k \|^{\frac{2p(1-\eta)}{p+2}}_{W^{-1,p}(\mathbb R^2)} + \left( \dfrac{a}{d}\right)^8\right).
\end{equation}
Plugging \eqref{eq_r2-partie1} and \eqref{eq_r2-partie-2} into \eqref{eq_r2_split}
yields finally:
\begin{equation}\label{r2N}
 \| r_{2} \|_{L^2(\mathcal{O}\cap \mathcal{F}_{N})} \leq C(\mathcal{O},p,\eta,\varepsilon_0,\K,K_{PM},R_{f}) M_{f} \Big( \Big(\frac ad\Big)^{4} + \|\mu - k \|^{\frac{p(1-\eta)}{p+2}}_{W^{-1,p}(\mathbb R^2)} \Big).
\end{equation}

\medskip 

Now we turn to deal with $r_{3}$. We split again:
\[
\|r_3\|^2_{L^2(\mathcal O \cap \mathcal F_N)}\leq 
\|r_3\|^2_{L^2(\mathcal O \setminus \bigcup B(x_{\ell},2a))} + \sum_{\ell=1}^N \|r_3\|^2_{L^2(B(x_{\ell},2a))}.
\]
Concerning the first term on the right-hand side, as for $r_{2}$, we use Lemma~\ref{V1-decay} and \eqref{scaling-V} to notice that: 
\[
r_{3}(x)= \sum_{j=2}^3 \sum_{\ell=1}^N V^{a}[A_{\ell}^{(j)}](x-x_\ell)
= \sum_{j=2}^3 \sum_{\ell=1}^N a^2\dfrac{(\M_{\K} A_{\ell}^{(j)})\cdot (x-x_\ell)}{|x-x_\ell|^2} + a^3A_{\ell}^{(j)}\cdot \frac{h_{0}(\frac{x-x_\ell}a)}{|x-x_\ell|^{2}}.
\]
For the first term, we apply that ${\rm dist}(x,\{x_1,\ldots,x_N\}) \geq 2a$ to bound:
\begin{align*}
 |r_{3}(x)| 
 \leq &\sum_{j=2}^3 C(\K) \sum_{\ell=1}^N a^2 \dfrac{|A_{\ell}^{(j)}|}{|x-x_\ell|}
 \leq \sum_{j=2}^3 C(\K) \int_{\R^2} \dfrac{\sum_{\ell=1}^N |A_{\ell}^{(j)}| \mathds{1}_{B(x_{\ell},a)}(y)}{|x-y|} \d y\\
 \leq & \sum_{j=2}^3 C(\K) \int_{\R^2} \Big(\sum_{\ell=1}^N |A_{\ell}^{(j)}| \mathds{1}_{B(x_{\ell},a)}(y)\Big) \dfrac{\mathds{1}_{|x-y|\leq R_{\mathcal O} }}{|x-y|} \d y,
 \end{align*}
where $R_{\mathcal O} = {\rm diam}(K_{PM} \cup \mathcal O)$ (since $x \in \mathcal O$ while $y \in K_{PM}$ in the above integral). So by the Young's convolution inequality, we get
\begin{align*}
\| r_{3}\|_{L^2(\mathcal{O} \setminus \bigcup B(x_{\ell},2a) )} 
&\leq C(\K) \Big\| \dfrac{\mathds{1}_{ |y|\leq R_{\mathcal O} }}{|y|} \Big\|_{L^1(\R^2)} \sum_{j=2}^3 \Big\| \sum_{\ell=1}^N |A_{\ell}^{(j)}| \mathds{1}_{B(x_{\ell},a)}(y) \Big\|_{L^2(\R^2)} \\
&\leq C(\K,\mathcal{O}, K_{PM})\sum_{j=2}^3 (a^2 N)^{\frac12-\frac1q} \Big\| \sum_{\ell=1}^N |A_{\ell}^{(j)}| \mathds{1}_{B(x_{\ell},a)}(y) \Big\|_{L^q(\R^2)} \\
&\leq C(\K,\mathcal{O}, K_{PM})\Big(\frac ad\Big)^{2(\frac12-\frac1q)} a^{\frac2q}\sum_{j=2}^3 \Big(\sum_{\ell=1}^N |A_{\ell}^{(j)}|^q\Big)^{1/q} .
 \end{align*}
Next, we use Lemma~\ref{Aapp} to state that 
\[
 \sum_{j=2}^3 \Big(\sum_{\ell=1}^N |A_{\ell}^{(j)}|^q\Big)^{1/q} \leq C(q,\varepsilon_{0}) \Big(\frac{a}{d}\Big)^{2(1-\frac{1}{q})} \Big(\sum_{\ell=1}^N |A^{(1)}_{\ell}|^q\Big)^{1/q}
\]
provided $\varepsilon_{0} <1/2$. So we conclude by \eqref{est-A1} and the fact that $N\leq C(K_{PM})d^{-2}$:
\begin{align*}
 \| r_{3}\|_{L^2(\mathcal{O} \setminus \bigcup B(x_{\ell},2a))} 
 &\leq C(q,\varepsilon_{0},\K, \mathcal{O},K_{PM})\Big(\frac ad\Big)^{2(\frac12-\frac1q)} a^{\frac2q} \Big(\frac{a}{d}\Big)^{2(1-\frac{1}{q})} R_{f}M_{f} \frac1{d^{2/q}}\\
 &\leq C(q,\varepsilon_{0},\K,\mathcal{O},K_{PM})R_{f}M_{f}\Big(\frac ad\Big)^{3-\frac2q} 
\end{align*}
and by taking $q$ sufficiently large, we reach:
\begin{equation} \label{eq_r3partie1}
 \| r_{3}\|_{L^2(\mathcal{O} \setminus \bigcup B(x_{\ell},2a))} \leq C(\mathcal O,\varepsilon_0,\K,K_{PM},\eta) R_{f}M_{f}\Big(\frac ad\Big)^{3-\eta}. 
\end{equation}

Concerning the remaining term $\sum_{\ell=1}^N \|r_3\|^2_{L^2(B(x_{\ell},2a))}$, we proceed as for $r_2$. On any $B(x_{\ell},2a)$ we have,
since ${\rm dist}(x,\{x_{\lambda}, \, \lambda \neq \ell \}) \geq d/2$:
\begin{align*}
|r_3(x)| & \leq C \sum_{j=2}^3 \left( a|A_{\ell}^{(j)}| + \sum_{\lambda \neq \ell} \dfrac{a^2 |A_{\lambda}^{(j)}|}{|x-x_{\lambda}|}\right)
\leq C \sum_{j=2}^{3} \left( a |A_{\ell}^{(j)}| + a^2 \left( \sum_{\lambda \neq \ell} \dfrac{1}{|x-x_{\lambda}|^{4/3}}\right)^{\frac 34} \left( \sum_{\lambda=1}^N |A_{\lambda}^{(j)}|^4\right)^{\frac 14} \right)\\
& \leq C(\mathcal{O}, K_{PM}) \sum_{j=2}^{3} \left( a |A_{\ell}^{(j)}| +\dfrac{a^2}{d^{3/2}} \left( \sum_{\lambda=1}^{N} |A_{\lambda}|^4 \right)^{1/4} \right).
\end{align*}
Consequently: 
\begin{align*}
\sum_{\ell=1}^N \|r_3\|^2_{L^2(B(x_{\ell},2a))}
& \leq C(\mathcal{O}, K_{PM}) a^2 \sum_{j=2}^3 \sum_{\ell=1}^N \left( a^2 |A_\ell^{(j)}|^2 + \dfrac{a^4}{d^4} \left( \frac 1N\sum_{\lambda=1}^{N} |A_{\lambda}^{(j)}|^4 \right)^{\frac 12}
\right)
\\
& \leq C(\mathcal{O},K_{PM}) \sum_{j=2}^{3} \left( a^2 \dfrac{a^2}{d^2} \dfrac{1}{N}\sum_{\ell=1}^N |A_\ell^{(j)}|^2 + \dfrac{a^6}{d^6} \left( \frac 1N\sum_{\lambda=1}^{N} |A_{\lambda}^{(j)}|^4 \right)^{\frac 12} \right).
\end{align*}
We apply then again Lemma~\ref{Aapp} with $q= 2$ and $q=4$ together with \eqref{psi0-bd} to obtain:
\begin{equation} \label{eq_r3partie2}
\sum_{\ell=1}^N \|r_3\|^2_{L^2(B(x_{\ell},2a))} \leq C(\mathcal{O},K_{PM},R_f) M_f^2 \left(a^2 \left(\dfrac{a}{d}\right)^{4} + \left( \dfrac{a}{d} \right)^{9} \right).
\end{equation}
Finally, combining \eqref{eq_r3partie1}-\eqref{eq_r3partie2}, we obtain, similarly as above:
\begin{equation}\label{r3N}
\|r_{3}\|_{L^2(\mathcal O \cap \mathcal F_N)} \leq C(\mathcal O,K_{PM},\varepsilon_0,\eta,R_f) M_{f} \left( \left(\dfrac ad \right)^{3-\eta} + \|\mu - k \|^{\frac{2p}{p+2}}_{W^{-1,p}(\mathbb R^2)} \right).
\end{equation}

Bringing together \eqref{r1N}-\eqref{r2N}-\eqref{r3N}, the proposition is proved.
\end{proof}

\subsection{End of proof of Theorem~\ref{main-elliptic}}

To finish the proof of Theorem~\ref{main-elliptic}, we first assume that 
\[
\varepsilon_0 \leq \min(1/4, \varepsilon_{ref}, \eps_c(3))
\]
 so that Corollary~\ref{1st term-bis}, Proposition~\ref{uc-existence} and Proposition~\ref{second term} hold true. We decompose (up to an additive constant) $\psi_N$ into:
\[
\psi_N-\psi_c=\psi_N-\bar{\psi}_N+\bar{\psi}_N-\tilde{\psi}_c+\tilde{\psi}_c-\psi_c
\]
where we recall that $\bar{\psi}_N = \psi^{(3)}$ (with the notation of Proposition~\ref{1st term}) and $\tilde{\psi}_c=
\psi_{0} - \Delta^{-1}\div(kM_{\K} \nabla \psi_{0})$.

\medskip

First according to Corollary~\ref{1st term-bis}, we have that
\[
\|\psi_N-\bar{\psi}_N\|_{\dot{H}^1(\mathcal{F}_N)}\leq C_{app}(R_{f},p,\eps_0) M_{f}\left( \Big(\frac{a}{d}\Big)^{4} +\|\mu - k\|^{\frac{p}{p+2}}_{W^{-1,p}(\mathbb R^2)}\right) . 
\]
By Remark~\ref{uc-expansion}, we have for any $\mathcal{O}\Subset \R^2$:
\[
\|\psi_c-\tilde{\psi}_c\|_{\dot{H}^1(\mathcal{O}\cap\mathcal{F}_N)}\leq C(\mathcal{O})\|\psi_c-\tilde{\psi}_c\|_{\dot{W}^{1,3}(\mathbb R^2)}\leq C(\mathcal{O},R_{f} )M_{f} \|k\|^2_{L^{\infty}(\mathbb R^2)} .
\]
Hence we obtain that $\Gamma_{1,N}:=\psi_N-\bar{\psi}_N+\tilde{\psi}_c-\psi_c$ is harmonic in $\mathbb R^2 \setminus K_{PM}$ and satisfies
\[
\|\nabla\Gamma_{1,N}\|_{L^2(\mathcal{O}\cap \mathcal{F}_N)}\leq C(\eps_0,p, \mathcal{O},R_{f}, K_{PM}, \K)M_{f} \left[
 \Big(\frac{a}{d}\Big)^{4} + \|\mu - k\|^{\frac{p}{p+2}}_{W^{-1,p}(\mathbb R^2)} + 
 \|k\|^2_{L^{\infty}(\mathbb R^2)} \right].
\]

On the other hand, by Proposition~\ref{second term}, we have that $\Gamma_{2,N} := \bar{\psi}_N-\tilde{\psi}_c$
is again harmonic in $\R^2\setminus K_{PM}$ and satisfies:
\begin{multline*}
\|\Gamma_{2,N}\|_{L^2(\mathcal{O}\cap \mathcal{F}_N)}\leq
C(\eps_0,p,\eta, \mathcal{O},R_{f}, K_{PM}, \K)M_{f} \\
\times \left[ \left( \dfrac{a}{d}\right)^{3-\eta} +
\|\mu -k\|^{ \frac{p(1-\eta)}{p+2}}_{W^{-1,p}(\mathbb R^2)}
 + \|\mu -k\|^{\frac12}_{W^{-1,p}(\mathbb R^2)} \right].
\end{multline*}

The theorem is finally proved.

\section{The homogenized Euler equations}\label{sec-Euler}

We split this section in two parts. The first subsection concerns the well-posedness result for the homogenized Euler equations \eqref{Eulerc}. The second subsection concerns the proof of the second statement of Theorem~\ref{main-Euler}. In this section, the subscript $c$ on functional set, as $C^1_{c}(U)$ or $W^{k,p}_{c}(U)$, means that the functions have a compact support in $U$.

\subsection{Well-posedness of the modified Euler equations} \label{Sect-WP}

We begin this section by recalling the main result of Proposition~\ref{uc-existence}: for any $\varepsilon_{0}\leq \varepsilon_{c}(4),$ $k\in \mathcal{FV}(\varepsilon_{0})$ and $f\in L^\infty_{c}(\R^2)$, we have a unique (up to a constant) solution $\psi_{c}\in \dot W^{1,4}(\R^2)$ of \eqref{def-psic}.

\medskip

Of course, to prove the well-posedness of strong solution to \eqref{Eulerc}, we need more regularity concerning $u_{c}:=\nabla^\perp \psi_{c}$. 
As the main theorem holds true only for $\omega_{c}$ in the exterior of the porous medium, it is enough to get a well-posedness before the vorticity $\omega_{c}$ reaches the support of $k$. Hence, instead to use sophisticated arguments concerning differential operators in divergence form, see for instance the monograph of Auscher and Tchamitchian \cite{Auscher}, the following lemma will be enough for our purpose.

\begin{lemma}\label{lem-regellip2}
Let $\varepsilon_{0}\leq \varepsilon_{c}(4)$. For any $\delta >0$, there exists $C(\delta)>0$ depending only on $\delta$ such that the following holds true. For all $k\in \mathcal{FV}(\varepsilon_{0})$ and $f$ a bounded function compactly supported in $\R^2$, $\nabla \psi_{c}$ is continuous on 
\[
\mathcal{F}_{\delta} := \Big\{ x\in \R^2, {\rm dist}(x, K_{PM} )> \delta \Big\}
\]
 and 
\[
 \| \nabla \psi_{c} \|_{L^\infty(\mathcal{F}_{\delta} )} \leq C(\delta) \| f \|_{L^1\cap L^\infty(\R^2)}.
 \]

Moreover, if $f\in W^{1,\infty}_{c}(\R^2)$, then $\nabla \psi_{c}$ belongs to $C^1(\mathcal{F}_{\delta} )$ and
 \[
 \| \nabla^2 \psi_{c} \|_{L^\infty(\mathcal{F}_{\delta} )} \leq C(\delta) \Big(1+ \| f \|_{L^1\cap L^\infty(\R^2)}+\| f \|_{ L^\infty(\R^2)} \ln (1 + \| \nabla f \|_{L^\infty(\R^2)})\Big).
 \]
\end{lemma}
\begin{proof}
As $\supp k\subset K_{PM}$, we simply notice that 
 \[
 \tilde \psi := \psi_{c}-\psi_{0}
 \]
 is harmonic on $\mathcal{F}_{\delta/2}$, hence by the mean-value theorem
 \[
 \| \nabla \tilde \psi \|_{W^{1,\infty}(\mathcal{F}_{\delta} )} \leq C(\delta) \| \nabla \psi_{c}- \nabla\psi_{0} \|_{L^4(\mathcal{F}_{\delta/2})}\leq C(\delta) \|f \|_{L^1\cap L^\infty (\R^2)} ,
 \]
 provided $\varepsilon_{0}\leq \varepsilon_{c}(4)$ (see Proposition~\ref{uc-existence} with $q=4$).
 
 Therefore, the conclusion of this lemma follows directly from the standard estimates of $\psi_{0}=\Delta^{-1} f$.
\end{proof}

We are now in position to adapt the classical proof for the Euler equations to get the well-posedness of \eqref{Eulerc}. The main idea is to introduce the characteristic curve along the flow and to use an iteration procedure based on the wellposedness of the linear transport equation. We refer to Marchioro and Pulvirenti \cite{MarPul} for this type of construction.
As this proof is related to classical arguments, we only summarize here the procedure, and we refer to \cite{ADL} (proof of Theorem 2.2. in Section 7.1) where all the details are included in the context of $C^1$ solutions.

 We fix an initial data $\omega_{0}\in C^1_{c} (\R^2\setminus K_{PM})$. For any $\delta >0$ and $R_{f}>0$, we introduce now the subspace $C_{\omega_0,\delta,R_{f}}\subset C_c^1([0,t_\delta]\times \R^2)$ as follows: a vortex density $\omega \in C_c^1([0,t_\delta]\times \R^2)$ belongs to $C_{\omega_0,\delta,R_{f}}$ if and only if
\begin{itemize}
\item $\|\omega \|_{L^1(\R^2)}=\|\omega_0\|_{L^1(\R^2)}$ and $\|\omega\|_{L^\infty(\R^2)}=\|\omega_0\|_{L^\infty(\R^2)}$, for every $t\in [0,t_\delta]$,
\item $\omega(0,x)=\omega_0(x)$, for every $x\in\R^2$,
\item $\supp \omega(t,\cdot)\subset \mathcal{F}_{\delta}\cap B(0,R_{f})$, for every $t\in [0,t_\delta]$.
\end{itemize}
Of course, we have to consider $\delta <{\rm dist}(\supp \omega_{0},K_{PM})$ and $R_{f}>{\rm diam}( \supp \omega_{0}):=\sup_{x\in \supp \omega_{0}}{|x|}$. The subspace $C_{\omega_0,\delta,R_{f}}$ inherits its topology from the metric of $C_c^1\left([0,t_\delta]\times\R^2\right)$.

For any function $\omega\in C_{\omega_{0},\delta,R_{f}}$, we conclude from Lemma~\ref{lem-regellip2} that the velocity $u:=\nabla^\perp \psi_{c}[\omega]$ associated to $\omega$ is uniformly bounded in $\mathcal{F}_{\delta}$ by $C(\delta,R_{f})\|\omega_{0}\|_{L^1\cap L^\infty}$. Therefore, any trajectory starting from $\supp \omega_{0}$ along the flow associated to $\nabla^\perp \psi_{c}[\omega]$ stays in $\mathcal{F}_{\delta}\cap B(0,R_{f})$ at least until
\[
 t_{\delta} := \min \Big(\frac{{\rm dist}(\supp \omega_{0},K_{PM} )- \delta}{C(\delta,R_{f})\|\omega_{0}\|_{L^1\cap L^\infty}}, \frac{R_{f}-{\rm diam}( \supp \omega_{0})}{C(\delta,R_{f})\|\omega_{0}\|_{L^1\cap L^\infty}} \Big).
\]
The main idea is to prove the well-poseness result on $[0, t_{\delta}]$, following the usual scheme.

First, we build an approximating sequence $(\omega_{n} )_{n\in \N}$ using a standard iteration procedure based on the wellposedness of the linear transport equation.

The first term is simply given by the constant function $\omega_0(t,x)=\omega_0(x)$, for all $(t,x)\in [0,t_\delta]\times \R^2$. Then, for each $\omega_n\in C_{\omega_0,\delta,R_{f}}$, the following term $\omega_{n+1}\in C_{\omega_0,\delta,R_{f}}$ is defined as the unique solution to the linear transport equation
\[
\left\{
\begin{aligned}
& \partial_{t} \omega_{n+1} + u_n\cdot \nabla \omega_{n+1} =0,\\
& \omega_{n+1}(t=0) = \omega_0,
\end{aligned}
\right.
\]
where the velocity flow $u_{n}$ is given by
\[
u_n= \nabla^{\perp} \psi_{c} \quad \text{where}\quad \div [({\rm I}_{2}+kM_{\K}) \nabla \psi_c] = \omega_{n}.
\]
Indeed, as $u_{n}$ is lipschitz on $\mathcal{F}_{\delta}$, for any $x\in \supp \omega_{0}$ there exists a unique characteristic curve $X_{n}(\cdot,x)\in C^1([0,t_1] ; \mathcal{F}_{\delta})$, i.e.\ the curve solving the differential equation
\[
\frac{\d X_n(s,x)}{\d s}=u_n(s,X_n(x,s)), \quad X_{n}(0,x)=x.
\]
In view of the definition of $t_{\delta}$, we can choose $t_{1}=t_{\delta}$. For any fixed $t\in [0,t_{\delta}]$, the mapping $x\mapsto X_{n}(t,x)$ is a $C^1$ diffeomorphism from $\supp \omega_{0}$ onto its image, preserving the Lebesgue measure. Its inverse $X_{n}(t,\cdot)^{-1}$ allows us to define the new vortex density as
\begin{equation}\label{def-omega-n}
\left\{
\begin{array}{ll}
 \omega_{n+1}(t,x) = \omega_{0}( X_{n}(t,\cdot)^{-1}(x)) , & \text{if } x\in X_{n}(t,\supp \omega_{0}),\\
 \omega_{n+1}(t,x) =0,& \text{otherwise,} 
\end{array}\right.
\end{equation}
which belongs to $C_{\omega_{0},\delta,R_{f}}$.

Second, we establish uniform $C^1$-bounds on this approximating sequence. To this end, we write the equation verified by $\partial_{x_i}\omega_{n+1}$ and we prove by using the flow map $X_{n}$ associated to $u_{n}$ that for any $[a,b]\subset [0,t_{\delta}]$ such that $b-a$ is small enough
\[
 \|\omega_{n+1}\|_{C^1([a,b]\times\R^2)} 
 \leq C_{0}\Big(1+\sup_{p\geq 0}\|\nabla\omega_{p}(a,\cdot)\|_{L^\infty(\R^2)}^2 \Big),
\]
 where $C_0>0$ may only depend on $\|\omega_0\|_{L^1\cap L^\infty(\R^2)}$, $k$, $\delta$, $R_{f}$ but is independent of $\omega_n$, $\omega_{n+1}$ and $[a,b]$. 
 
We deduce that we may propagate the preceding $C^1$-bound on $[a,b]$ to the whole interval $[0,t_\delta]$. This yields a uniform bound
\begin{equation}\label{C1 bound}
\sup_{n\geq 0} \|\omega_{n}\|_{C^1([0,t_\delta]\times\R^2)}<\infty.
\end{equation}

Next, we show that $(\omega_{n} )_{n\in \N}$ is actually a Cauchy sequence 
 in $C^0([0,t_\delta]\times\R^2)$ and, therefore, there exists $\omega_{c}\in C([0,t_\delta]\times\R^2)$ such that
\begin{equation}\label{convergence xi mu}
\begin{aligned}
\omega_n & \longrightarrow \omega_{c}
\quad\text{in }L^\infty([0,t_\delta]\times\R^2),
\\
u_n & \longrightarrow u_{c}
\quad\text{in }L^\infty([0,t_\delta]\times \mathcal{F}_{\delta}),
\end{aligned}
\end{equation}
where $u_{c}$ is defined by $\nabla^\perp \psi_{c}[\omega]$ and we have used Lemma~\ref{lem-regellip2} to derive the convergence of $u_n$ from that of $\omega_n$. 
It is then readily seen that $\omega_{c}$ solves \eqref{Eulerc} in the sense of distributions.

In order to complete the proof of wellposedness in $C^1([0,t_{\delta}])$, there only remains to show that $\omega_{c}$ is actually of class $C^1$. Indeed, the uniqueness of solutions will then easily ensue from an estimate similar to the previous step. 

Using \eqref{convergence xi mu}, we can show that the characteristic curve $X_n$ associated to $u_{n}$ converges uniformly in $(t,x)\in [0,t_\delta]\times\supp\omega_0$ towards $X$ the characteristic curve associated to $u_{c}$.

 By the uniform convergences of $\omega_{n}$ to $\omega$ and $X_{n}$ to $X$, we conclude from \eqref{def-omega-n} and \eqref{C1 bound} that
\begin{equation*}
\left\{
\begin{aligned}
\omega_{c}(t,x) & = \omega_0(X^{-1}(t,x)), && \text{if } x\in X(t,\supp\omega_0),
\\
\omega_{c}(t,x) & = 0, && \text{otherwise},
\end{aligned}
\right.
\end{equation*}
which establishes that $\omega_{c} \in C^1_c([0,t_\delta]\times\R^2)$.

If $\dist(\supp \omega_{c}(t_{\delta},\cdot), \mathcal{F}_{\delta}^c\cup B(0,R_{f})^c >0)$, we iterate our construction until we get the well-posedeness on $[0,T_{k}]$ where $\dist(\supp \omega_{c}(T_{k},\cdot), \mathcal{F}_{\delta}^c\cup B(0,R_{f})^c)=0$ (or $T_{k}=+\infty$). 

Considering a sequence $R_{f}$ which tends to infinity, this ends the proof of the first statement in Theorem~\ref{main-Euler}, provided $\varepsilon_{0}\leq \varepsilon_{c}(4)$ (due to Lemma~\ref{lem-regellip2}).

Actually, we can even prove a strongest version: considering a sequence $\delta_{n}\to 0$, we get the well-posedeness on $[0,T_{*})$ such that $\dist(\supp \omega_{c}(t,\cdot), K_{PM})>0$ for any $t\in [0,T_{*})$, where $T_*=+\infty$ or $\dist(\supp \omega_{c}(t_{\delta},\cdot), K_{PM}) \to 0$ when $t\to T_*<+\infty$.

\subsection{Stability estimate}

This subsection is dedicated to the proof of the second statement of Theorem~\ref{main-Euler}. So let us consider $T\in (0,T_{k}]$, $\eta \in (0,1)$, $p\in (1,2)$ given. If $T_{k}<+\infty$, we can choose $T=T_{k}$. We are looking for restrictions on $\varepsilon_{0}$ such that the second statement holds true. For sure, we consider $\varepsilon_{0}\leq \tilde \varepsilon_{0}$ where $\tilde \varepsilon_{0}$ is the quantity $\varepsilon_{0}$ appearing in Theorem~\ref{main-elliptic}.

In the previous subsection, we note that the unique solution $\omega_{c}$ of \eqref{Eulerc} is such that 
\[
\| \omega_{c} (t,\cdot) \|_{L^\infty(\R^2)} = \| \omega_{0} \|_{L^\infty(\R^2)}, \quad \| \omega_{c} (t,\cdot) \|_{L^1(\R^2)} = \| \omega_{0} \|_{L^1(\R^2)} , \quad\supp \omega_{c} (t,\cdot) \subset \overline{\mathcal{F}_{\delta}} \quad\text{ for all }t\in [0;T],
\]
where we recall that $\mathcal{F}_{\delta}$ is defined in Lemma~\ref{lem-regellip2}. This lemma states that $\nabla \psi_{c}$ is uniformly bounded in $\mathcal{F}_{\delta}$ by $C(\delta)\|\omega_{0}\|_{L^1\cap L^\infty}$, independently of $k$ (provided $\varepsilon_{0}\leq \varepsilon_{c}(4)$). Hence there exists $R_{T}>0$ independent of $k$ such that
\[
\supp \omega_{c} (t,\cdot) \subset K_{T} \text{ for all }t\in [0;T], \text{ with } K_{T}:=\overline{B(0,R_{T})} \cap \overline{\mathcal{F}_{\delta}} .
\]
We also introduce a Jordan domain $\mathcal{O}_{T}$ such that $K_{T}\subset \mathcal{O}_{T} \subset \mathcal{F}_{\delta/2}$ and ${\rm dist}(K_{T},\partial \mathcal{O}_{T})\geq \delta/2$.

We finally set $M_{f}=\| \omega_{0} \|_{L^\infty(\R^2)}$ and $R_{f}$ large enough such that $\overline{\mathcal{O}_{T}}\subset B(0,R_{f})$. Hence, $M_{f}$ and $R_{f}$ are independent of $k$.

For any $f\in C^1_{c}(\R^2)$, we note $\psi_{c}[f]$ and $\psi_{N}[f]$ respectively the solution of \eqref{def-psic} and \eqref{def-psiN}, then we derive the following corollary from Theorem~\ref{main-elliptic}.
\begin{corollary}\label{cor-elliptic}
 There exists $C$ such that for any $\varepsilon_{0}\leq \tilde \varepsilon_{0}$ (where $\tilde \varepsilon_{0}$ is the quantity $\varepsilon_{0}$ appearing in Theorem~\ref{main-elliptic}), any $k \in \mathcal{FV}(\varepsilon_0)$, any $\mathcal{F}_{N}$ verifying \eqref{main_assumption} and any $f$ compactly supported in $B(0,R_{f})$ which is bounded by $M_{f}$, then 
\[ 
 \| \nabla \psi_{N}[f]- \nabla \psi_{c}[f] \|_{W^{1,\infty}(\mathcal{O}_{T})} 
 \leq 
 C
\left[ \left( \dfrac{a}{d}\right)^{3-\eta} +
\|\mu -k\|^{ \frac{p(1-\eta)}{p+2}}_{W^{-1,p}(\mathbb R^2)}
 + \|\mu -k\|^{\frac12}_{W^{-1,p}(\mathbb R^2)} + \|k\|_{L^{\infty}(\mathbb R^2)}^2
 \right].
 \]
\end{corollary}

The proof comes directly from the mean value theorem (and the harmonicity of $\Gamma_{j,N}$), because $B(x,\delta/2)\subset \R^2 \setminus K_{PM}\subset \mathcal{F}_{N}$ for all $x\in \mathcal{O}_{T}$.

Next, we define $T_{N}\in (0, T]$ such that $\omega_{N}$ stays compactly supported in $\overline{\mathcal{O}_{T}}$:
\[
 T_{N}:= \sup_{T_{*}\in [0,T]} \Big\{ T_{*}, \ \supp \omega_{N}(t,\cdot)\subset \overline{\mathcal{O}_{T}} \ \forall t\in [0,T_{*}] \Big\}.
\]
As the vorticity is transported by a continuous vector field $u_{N}$, we state that $T_{N}>0$ and that there are only two possibilities:
\begin{enumerate}
 \item[(i)] $T_{N}=T$, hence $\supp \omega_{N}(t,\cdot)\subset \overline{\mathcal{O}_{T}}$ for all $t\in [0,T]$ ;
 \item[(ii)] $T_{N}<T$, hence $\overline{\supp \omega_{N}(T_{N},\cdot)} \cap \partial \mathcal{O}_{T} \neq \emptyset$. \label{pageTN}
\end{enumerate}
In the sequel of this section, we will derive uniform estimates for all $t\in [0,T_{N}]$ (where the support of $\omega_{N}$ is included in $\overline{\mathcal{O}_{T}}$), and we will conclude by a bootstrap argument that (ii) cannot happen if $\varepsilon_{0}$ is chosen small enough, which will imply that the estimates hold true on $[0,T]$.

From the solution $(u_{c},\omega_{c})$, we can define the trajectories $(t,x)\mapsto X_{c}(t,x)$ on $\R^+\times \supp \omega_{0}$ by
\begin{equation}\label{def-Xc}
 \left\{
\begin{array}{l}
 \frac{\partial X_c}{\partial t}(t,x) = u_{c}(t,X_{c}(t,x)), \\
 X_{c}(0,x)=x,
\end{array}
 \right.
\end{equation}
and we recall that the vorticity is constant along the trajectories: $\omega_{c}(t,X_{c}(t,x))=\omega_{0}(x)$. Hence, for any $x\in \supp \omega_{0}$, $X_{c}(t,x)\in K_{T}$ for all $t\in [0,T]$. 
In the same way, we define the trajectories associated to $(u_{N},\omega_{N})$ on $\R^+\times \mathcal{F}_{N}$ by
\begin{equation}\label{def-XN}
 \left\{
\begin{array}{l}
 \frac{\partial X_{N}}{\partial t}(t,x) = u_{N}(t,X_{N}(t,x)), \\
 X_{N}(0,x)=x,
\end{array}
 \right.
\end{equation}
along of which $\omega_{N}$ is constant, and $X_{N}(t,x)\in \overline{\mathcal{O}_{T}}$ for all $(t,x)\in [0,T_{N}]\times \supp \omega_{0}$.

\subsubsection{Stability estimate for velocities}\label{sec-stab-vel}
The first step of our proof is to derive a uniform estimate of $u_{c}-u_{N}$ in $[0,T_{N}]\times \mathcal{O}_{T}$. Introducing the vector field $\check u_{N}:= \nabla^\perp \psi_{c}[\omega_{N}]$ and as $\omega_{N}(t,\cdot)$ is compactly supported in $B(0,R_{f})$ with $\|\omega_{N}(t,\cdot)\|_{L^\infty}=M_{f}$, Corollary~\ref{cor-elliptic} states
 \[
 \| (\check u_{N} - u_{N})(t,\cdot) \|_{L^{\infty}(\mathcal{O}_{T})} \leq C
\left[ \left( \dfrac{a}{d}\right)^{3-\eta} +
\|\mu -k\|^{ \frac{p(1-\eta)}{p+2}}_{W^{-1,p}(\mathbb R^2)}
 + \|\mu -k\|^{\frac12}_{W^{-1,p}(\mathbb R^2)} + \|k\|_{L^{\infty}(\mathbb R^2)}^2
 \right], 
 \]
for all $t\in [0,T_{N}]$. For shortness, we denote in the sequel of the proof:
\[
F(N,k):=\left( \dfrac{a}{d}\right)^{3-\eta} +
\|\mu -k\|^{ \frac{p(1-\eta)}{p+2}}_{W^{-1,p}(\mathbb R^2)}
 + \|\mu -k\|^{\frac12}_{W^{-1,p}(\mathbb R^2)} + \|k\|_{L^{\infty}(\mathbb R^2)}^2.
\] 
The second part can be estimated by Lemma~\ref{lem-regellip2}:
\begin{align*}
 \| (u_{c} - \check u_{N} (t,\cdot) \|_{L^{\infty}(\mathcal{O}_{T})}& \leq \| \nabla \psi_{c}[\omega_{c}-\omega_{N}] (t,\cdot) \|_{L^{\infty}(\mathcal{F}_{\delta/2})} \leq C(\delta/2) \|( \omega_{c}-\omega_{N})(t,\cdot) \|_{L^1\cap L^\infty(\R^2)} \\
 &\leq C \|( \omega_{c}-\omega_{N})(t,\cdot)\|_{L^\infty(\R^2)}, \qquad \forall t\in [0,T_{N}],
\end{align*}
because $\omega_{c}$ and $\omega_{N}$ are supported in $B(0,R_{f})$.

Putting together these two estimates, we conclude that
\begin{equation}\label{stab-vel}
 \| (u_{c} - u_{N})(t,\cdot) \|_{L^{\infty}(\mathcal{O}_{T})} \leq C\Big[ F(N,k) + \|( \omega_{c}-\omega_{N})(t,\cdot)\|_{L^\infty(\R^2)}\Big], \qquad \forall t\in [0,T_{N}].
 \end{equation}

Moreover, thanks to the second estimate of Lemma~\ref{lem-regellip2}, we state that
\begin{align*}
\| \nabla \check u_{N}(t,\cdot) \|_{L^\infty(\mathcal{O}_{T})}& \leq C \Big(1+ \| \omega_{N}(t,\cdot) \|_{L^1\cap L^\infty(\R^2)} + \| \omega_{N}(t,\cdot) \|_{ L^\infty(\R^2)} \ln(1+ \| \nabla \omega_{N}(t,\cdot) \|_{L^\infty(\R^2)})\Big)\\
&\leq C(\delta,\omega_{0}) \ln(2+ \| \nabla \omega_{N}(t,\cdot) \|_{L^\infty(\R^2)}) 
\end{align*}
so Corollary~\ref{cor-elliptic} implies that there exists $C$ such that
\begin{equation}\label{est-vel}
 \| \nabla u_{N}(t,\cdot) \|_{L^{\infty}(\mathcal{O}_{T})} \leq C \ln(2+ \| \nabla \omega_{N}(t,\cdot) \|_{L^\infty(\R^2)}) , \qquad \forall t\in [0,T_{N}].
\end{equation}

\subsubsection{Uniform $C^1$ estimates for vorticities}

Differentiating the vorticity equation, we get for $i=1,2$:
\[
\partial_{t} \partial_{x_{i}} \omega_{N} + u_{N}\cdot \nabla \partial_{x_{i}} \omega_{N}=-\partial_{x_{i}} u_{N} \cdot \nabla \omega_{N},
\]
hence
\[
\partial_{x_{i}} \omega_{N}(t,X_{N}(t,x)) = \partial_{x_{i}} \omega_{0}(x) -\int_{0}^t (\partial_{x_{i}} u_{N} \cdot \nabla \omega_{N})(s,X_{N}(s,x))\d s.
\]
As $X_{N}(t,x)\in \overline{\mathcal{O}_{T}}$ for all $(t,x)\in [0,T_{N}]\times \supp \omega_{0}$ and as we know that $\| \nabla u_{N} \|_{L^\infty([0,T_{N}]\times \mathcal{O}_{T})}$ is bounded (see \eqref{est-vel}), we get that
\[
\| \nabla \omega_{N}(t,\cdot) \|_{L^\infty(\R^2)} \leq \| \nabla \omega_{0} \|_{L^\infty(\R^2)} + C \int_{0}^t \| \nabla \omega_{N}(s,\cdot) \|_{L^\infty(\R^2)}\ln(2+ \| \nabla \omega_{N}(s,\cdot) \|_{L^\infty(\R^2)}) \d s.
\]
Gronwall's lemma allows us to conclude the following estimate for the vorticity:
\begin{equation}\label{est-vort}
 \| \nabla \omega_{N}(t,\cdot) \|_{L^{\infty}(\R^2)} \leq C , \qquad \forall t\in [0,T_{N}].
\end{equation}

\subsubsection{Stability estimate for vorticities}

Subtracting the vorticity equations, we can write
\begin{gather*}
 \partial_{t}(\omega_{c}-\omega_{N}) + u_{c}\cdot \nabla (\omega_{c}-\omega_{N})=-(u_{c}-u_{N})\cdot \nabla \omega_{N},\\
 \partial_{t}(\omega_{c}-\omega_{N}) + u_{N}\cdot \nabla (\omega_{c}-\omega_{N})=-(u_{c}-u_{N})\cdot \nabla \omega_{c}
\end{gather*}
which imply that
\begin{gather*}
(\omega_{c}-\omega_{N})(t,X_{c}(t,x)) =- \int_{0}^t \Big( (u_{c}-u_{N})\cdot \nabla \omega_{N}\Big)(s,X_{c}(s,x))\d s,\\
(\omega_{c}-\omega_{N})(t,X_{N}(t,x)) =- \int_{0}^t \Big( (u_{c}-u_{N})\cdot \nabla \omega_{c}\Big)(s,X_{N}(s,x))\d s.
\end{gather*}
As the support of $(\omega_{c}-\omega_{N})(t,\cdot)$ is included in $X_{c}(t,\supp \omega_{0}) \cup X_{N}(t,\supp \omega_{0})$, we use \eqref{stab-vel} and \eqref{est-vort} to write
\[
 \| ( \omega_{c}-\omega_{N})(t,\cdot) \|_{L^{\infty}(\R^2)} \leq C\Big[ \Big( F(N,k) + \int_{0}^t \|( \omega_{c}-\omega_{N})(s,\cdot)\|_{L^\infty(\R^2)} \d s \Big], \qquad \forall t\in [0,T_{N}].
\]

Therefore, Gronwall's lemma gives
\begin{equation}\label{stab-vort}
 \| ( \omega_{c}-\omega_{N})(t,\cdot) \|_{L^{\infty}(\R^2)} \leq C F(N,k) , \qquad \forall t\in [0,T_{N}],
 \end{equation}
and \eqref{stab-vel} becomes
\begin{equation}\label{stab-vel2}
 \| (u_{c} - u_{N})(t,\cdot) \|_{L^{\infty}(\mathcal{O}_{T})} \leq C F(N,k) , \qquad \forall t\in [0,T_{N}].
 \end{equation}

\subsubsection{Stability estimate for trajectories}

From the definition of the trajectories \eqref{def-Xc}-\eqref{def-XN} and repeating the decomposition of Section~\ref{sec-stab-vel}, we compute
\begin{align*}
 \partial_{t} |(X_{N}-X_{c})(t,x) |^2 &\leq 2|(X_{N}-X_{c})(t,x) | \Big( |(u_{N}-u_{c})(t,X_{N}(t,x)) | + |u_{c} (t,X_{N}(t,x) ) -u_{c}(t,X_{c}(t,x)) | \Big)\\
 &\leq C|(X_{N}-X_{c})(t,x) | \Big( F(N,k) + |(X_{N}-X_{c})(t,x) | \Big), \qquad \forall t\in [0,T_{N}],
\end{align*}
where we have used \eqref{stab-vel2} and that $u_{c}\in C^1([0,T]\times \mathcal{O}_{T})$. We deduce again by Gronwall's lemma that 
\begin{equation}\label{stab-traj}
 | (X_{N}-X_{c})(t,x) | \leq C F(N,k) , \qquad \forall t\in [0,T_{N}],\ \forall x\in \supp \omega_{0}.
 \end{equation}

\subsubsection{Bootstrap argument and conclusion}

To summarize, for $\delta$ given, we choose $\varepsilon_{0}\leq \tilde \varepsilon_{0}$ such that
\begin{equation}\label{bootstrap-cond}
C F(N,k) < \delta/2 
\end{equation}
for any $k \in \mathcal{FV}(\varepsilon_0)$ and $\mathcal{F}_{N}$ verifying \eqref{main_assumption}, where $C$ is the constant appearing in \eqref{stab-traj}. We point out that $C$ depends on $\tilde\varepsilon_{0}$ but not on $\varepsilon_{0}$. As $\varepsilon_{0}\leq \tilde \varepsilon_{0}$, for any $k \in \mathcal{FV}(\varepsilon_0)$ and $\mathcal{F}_{N}$ verifying \eqref{main_assumption}, the estimates \eqref{stab-vort}-\eqref{stab-traj} are valid. As $X_{c}(t,x)\in K_{T}$ for all $(t,x)\in [0,T]\times \supp \omega_{0}$, we conclude from \eqref{stab-traj} and \eqref{bootstrap-cond} that the situation (ii) in Page \pageref{pageTN} is impossible. This allows us to conclude that $T_{N}=T$ and that \eqref{stab-vort}-\eqref{stab-traj} are valid for all $t\in [0,T]$.

In Section~\ref{sec-stab-vel}, replacing $\mathcal{O}_{T}$ by any bounded open set $\mathcal{O}\Subset \R^2\setminus K_{PM}$, and using \eqref{stab-vort}, we get easily that \eqref{stab-vel2} is valid if we replace $\mathcal{O}_{T}$ by $\mathcal{O}$. This ends the proof of Theorem~\ref{main-Euler}.

\medskip

\noindent
{\bf Acknowledgements.} The two first authors are partially supported by the Agence Nationale de la Recherche, project IFSMACS, grant ANR-15-CE40-0010 and Project SINGFLOWS, grant ANR-18-CE40-0027-01. 

\def\cprime{$'$}


\begin{thebibliography}{10}

\bibitem{Allaire90a}
G.~Allaire.
\newblock Homogenization of the {N}avier-{S}tokes equations in open sets
  perforated with tiny holes. {I}. {A}bstract framework, a volume distribution
  of holes.
\newblock {\em Arch. Rational Mech. Anal.}, 113(3):209--259, 1990.

\bibitem{Allaire90b}
G.~Allaire.
\newblock Homogenization of the {N}avier-{S}tokes equations in open sets
  perforated with tiny holes. {II}. {N}oncritical sizes of the holes for a
  volume distribution and a surface distribution of holes.
\newblock {\em Arch. Rational Mech. Anal.}, 113(3):261--298, 1990.

\bibitem{Allaire91}
G.~Allaire.
\newblock Homogenization of the {N}avier-{S}tokes equations with a slip
  boundary condition.
\newblock {\em Comm. Pure Appl. Math.}, 44(6):605--641, 1991.

\bibitem{ADL}
D.~Ars\'{e}nio, E.~Dormy, and C.~Lacave.
\newblock The {V}ortex {M}ethod for {T}wo-{D}imensional {I}deal {F}lows in
  {E}xterior {D}omains.
\newblock {\em SIAM J. Math. Anal.}, 52(4):3881--3961, 2020.

\bibitem{Auscher}
P.~Auscher and P.~Tchamitchian.
\newblock Square root problem for divergence operators and related topics.
\newblock {\em Ast\'{e}risque}, (249):viii+172, 1998.

\bibitem{BaoLiYin}
E.~S. Bao, Y.~Y. Li, and B.~Yin.
\newblock Gradient estimates for the perfect conductivity problem.
\newblock {\em Arch. Ration. Mech. Anal.}, 193(1):195--226, 2009.

\bibitem{BLM}
V.~Bonnaillie-No{\"e}l, C.~Lacave, and N.~Masmoudi.
\newblock Permeability through a perforated domain for the incompressible 2{D}
  {E}uler equations.
\newblock {\em Ann. Inst. H. Poincar\'e Anal. Non Lin\'eaire}, 32(1):159--182,
  2015.

\bibitem{Bonnetier}
E.~Bonnetier, C.~Dapogny, and F.~Triki.
\newblock Homogenization of the eigenvalues of the {N}eumann-{P}oincar\'{e}
  operator.
\newblock {\em Arch. Ration. Mech. Anal.}, 234(2):777--855, 2019.

\bibitem{Diaz99}
J.~I. D{\'\i}az.
\newblock Two problems in homogenization of porous media.
\newblock In {\em Proceedings of the Second International Seminar on Geometry,
  Continua and Microstructure (Getafe, 1998)}, volume~14, pages 141--155, 1999.

\bibitem{FeireislLu}
E.~Feireisl and Y.~Lu.
\newblock Homogenization of stationary {N}avier-{S}tokes equations in domains
  with tiny holes.
\newblock {\em J. Math. Fluid Mech.}, 17(2):381--392, 2015.

\bibitem{GVH}
D.~G{\'e}rard-Varet and M.~Hillairet.
\newblock Analysis of the viscosity of dilute suspensions beyond einstein's
  formula.
\newblock {\em Arch. Ration. Mech. Anal.}, 2020.
\newblock Online First.

\bibitem{Gloria}
A.~Gloria.
\newblock A scalar version of the caflisch-luke paradox.
\newblock arXiv:1907.08182, 2019.

\bibitem{HillairetWu}
M.~Hillairet and D.~Wu.
\newblock Effective viscosity of a polydispersed suspension.
\newblock {\em J. Math. Pures Appl. (9)}, 138:413--447, 2020.

\bibitem{Hofer-Velazquez}
R.~M. H\"{o}fer and J.~J.~L. Vel\'{a}zquez.
\newblock The method of reflections, homogenization and screening for {P}oisson
  and {S}tokes equations in perforated domains.
\newblock {\em Arch. Ration. Mech. Anal.}, 227(3):1165--1221, 2018.

\bibitem{ILL}
D.~Iftimie, M.~C. Lopes~Filho, and H.~J. Nussenzveig~Lopes.
\newblock Two dimensional incompressible ideal flow around a small obstacle.
\newblock {\em Comm. Partial Differential Equations}, 28(1-2):349--379, 2003.

\bibitem{Kikuchi}
K.~Kikuchi.
\newblock Exterior problem for the two-dimensional {E}uler equation.
\newblock {\em J. Fac. Sci. Univ. Tokyo Sect. IA Math.}, 30(1):63--92, 1983.

\bibitem{LM}
C.~Lacave and N.~Masmoudi.
\newblock Impermeability through a perforated domain for the incompressible two
  dimensional {E}uler equations.
\newblock {\em Arch. Ration. Mech. Anal.}, 221(3):1117--1160, 2016.

\bibitem{LaMa}
C.~Lacave and A.~L. Mazzucato.
\newblock The vanishing viscosity limit in the presence of a porous medium.
\newblock {\em Math. Ann.}, 365(3-4):1527--1557, 2016.

\bibitem{LaurentLegendreSalomon}
P.~Laurent, G.~Legendre, and J.~Salomon.
\newblock {On the method of reflections}.
\newblock {\em HAL preprint}, 2017.
\newblock hal-01439871.

\bibitem{LionsMasmoudi}
P.-L. Lions and N.~Masmoudi.
\newblock Homogenization of the {E}uler system in a 2{D} porous medium.
\newblock {\em J. Math. Pures Appl. (9)}, 84(1):1--20, 2005.

\bibitem{Lu}
Y.~Lu and S.~Schwarzacher.
\newblock Homogenization of the compressible {N}avier-{S}tokes equations in
  domains with very tiny holes.
\newblock {\em J. Differential Equations}, 265(4):1371--1406, 2018.

\bibitem{MajdaBertozzi}
A.~J. Majda and A.~L. Bertozzi.
\newblock {\em Vorticity and incompressible flow}, volume~27 of {\em Cambridge
  Texts in Applied Mathematics}.
\newblock Cambridge University Press, Cambridge, 2002.

\bibitem{MarPul}
C.~Marchioro and M.~Pulvirenti.
\newblock {\em Mathematical theory of incompressible nonviscous fluids},
  volume~96 of {\em Applied Mathematical Sciences}.
\newblock Springer-Verlag, New York, 1994.

\bibitem{Masmoudi02esaim}
N.~Masmoudi.
\newblock Homogenization of the compressible {N}avier-{S}tokes equations in a
  porous medium.
\newblock {\em ESAIM Control Optim. Calc. Var.}, 8:885--906 (electronic), 2002.
\newblock A tribute to J. L. Lions.

\bibitem{Mikelic91}
A.~Mikeli{\'c}.
\newblock Homogenization of nonstationary {N}avier-{S}tokes equations in a
  domain with a grained boundary.
\newblock {\em Ann. Mat. Pura Appl. (4)}, 158:167--179, 1991.

\bibitem{MikelicPaoli}
A.~Mikeli{\'c} and L.~Paoli.
\newblock Homogenization of the inviscid incompressible fluid flow through a
  {$2$}{D} porous medium.
\newblock {\em Proc. Amer. Math. Soc.}, 127(7):2019--2028, 1999.

\bibitem{Niethammer-Schubert}
B.~Niethammer and R.~Schubert.
\newblock A local version of {E}instein's formula for the effective viscosity
  of suspensions.
\newblock {\em SIAM J. Math. Anal.}, 52(3):2561--2591, 2020.

\bibitem{Sanchez80}
E.~S{\'a}nchez-Palencia.
\newblock {\em Nonhomogeneous media and vibration theory}.
\newblock Springer-Verlag, Berlin, 1980.

\bibitem{Schubert}
R.~Schubert.
\newblock On the effective properties of suspensions.
\newblock Master's thesis, Universit\"at Bonn, 2018.

\bibitem{Stein}
E.~M. Stein.
\newblock {\em Singular integrals and differentiability properties of
  functions}.
\newblock Princeton Mathematical Series, No. 30. Princeton University Press,
  Princeton, N.J., 1970.

\bibitem{Tartar80}
L.~Tartar.
\newblock Incompressible fluid flow in a porous medium: convergence of the
  homogenization process.
\newblock {\em {\it in} Nonhomogeneous media and vibration theory (E.
  S{\'a}nchez-Palencia)}, pages 368--377, 1980.

\end{thebibliography}
\end{document}